\documentclass[letterpaper,11pt,reqno]{amsart}
\usepackage{indentfirst} 
\usepackage{amssymb}
\usepackage{mathtools} 
\usepackage{mathabx} 
\usepackage{amsthm}
\usepackage{thmtools}
\usepackage{enumitem} 
\usepackage[colorlinks=true]{hyperref} 
\usepackage[usenames,dvipsnames]{xcolor} 
\usepackage{ifthen}
\usepackage{tikz}
\usetikzlibrary{decorations.pathreplacing}
\usepackage{tikz-cd}
\usepackage[left=3.5cm, right=3.5cm, bottom=3.4cm]{geometry} 
  
\makeatletter 

\def\MRbibitem{\@ifnextchar[\my@lbibitem\my@bibitem}

\def\mybiblabel#1#2{\@biblabel{{\hyperref{http://www.ams.org/mathscinet-getitem?mr=#1}{}{}{#2}}}}

\def\myhyperanchor#1{\Hy@raisedlink{\hyper@anchorstart{cite.#1}\hyper@anchorend}}

\def\my@lbibitem[#1]#2#3#4\par{%
  \item[\mybiblabel{#2}{#1}\myhyperanchor{#3}\hfill]#4%
  \@ifundefined{ifbackrefparscan}{}{\BR@backref{#3}}%
  \if@filesw{\let\protect\noexpand\immediate
    \write\@auxout{\string\bibcite{#3}{#1}}}\fi\ignorespaces%
}

\def\my@bibitem#1#2#3\par{%
  \refstepcounter\@listctr
  \item[\mybiblabel{#1}{\the\value\@listctr}\myhyperanchor{#2}\hfill]#3%
  \@ifundefined{ifbackrefparscan}{}{\BR@backref{#2}}%
  \if@filesw\immediate\write\@auxout
    {\string\bibcite{#2}{\the\value\@listctr}}\fi\ignorespaces%
}

\makeatother

\DeclareFontFamily{U} {MnSymbolA}{}
\DeclareFontShape{U}{MnSymbolA}{m}{n}{
   <-6> MnSymbolA5
   <6-7> MnSymbolA6
   <7-8> MnSymbolA7
   <8-9> MnSymbolA8
   <9-10> MnSymbolA9
   <10-12> MnSymbolA10
   <12-> MnSymbolA12}{}
\DeclareFontShape{U}{MnSymbolA}{b}{n}{
   <-6> MnSymbolA-Bold5
   <6-7> MnSymbolA-Bold6
   <7-8> MnSymbolA-Bold7
   <8-9> MnSymbolA-Bold8
   <9-10> MnSymbolA-Bold9
   <10-12> MnSymbolA-Bold10
   <12-> MnSymbolA-Bold12}{}
\DeclareSymbolFont{MnSyA} {U} {MnSymbolA}{m}{n}
 \DeclareFontFamily{U} {MnSymbolC}{}
\DeclareFontShape{U}{MnSymbolC}{m}{n}{
  <-6> MnSymbolC5
  <6-7> MnSymbolC6
  <7-8> MnSymbolC7
  <8-9> MnSymbolC8
  <9-10> MnSymbolC9
  <10-12> MnSymbolC10
  <12-> MnSymbolC12}{}
\DeclareFontShape{U}{MnSymbolC}{b}{n}{
  <-6> MnSymbolC-Bold5
  <6-7> MnSymbolC-Bold6
  <7-8> MnSymbolC-Bold7
  <8-9> MnSymbolC-Bold8
  <9-10> MnSymbolC-Bold9
  <10-12> MnSymbolC-Bold10
  <12-> MnSymbolC-Bold12}{}
\DeclareSymbolFont{MnSyC} {U} {MnSymbolC}{m}{n}

\DeclareMathSymbol{\top}{\mathord}{MnSyA}{219} 
\DeclareMathSymbol{\plus}{\mathord}{MnSyC}{20} 

\declaretheorem[numberwithin=section]{theorem} 
\declaretheorem[sibling=theorem]{proposition} 
\declaretheorem[sibling=theorem]{lemma}
\declaretheorem[sibling=theorem]{corollary}

\declaretheorem[sibling=theorem, style=definition]{example}
\declaretheorem[sibling=theorem, style=definition]{definition}
\declaretheorem[sibling=theorem]{remark}

\hypersetup{bookmarksdepth = 3} 
\numberwithin{equation}{section}     

\setlist[enumerate,1]{label={\upshape(\alph*)},ref=\alph*}
\setlist[enumerate,2]{label={\upshape(\arabic*)},ref=\arabic*} 

\newcommand{\M}{\mathcal{M}}
\newcommand{\R}{\mathbb{R}}
\newcommand{\Z}{\mathbb{Z}}
\newcommand{\N}{\mathbb{N}}

\def\phi{\varphi}
\def\R{{\mathbb R}}

\def\N{{\mathbb N}}
\def\Z{{\mathbb Z}}

\def\F{{\mathcal F}}

\def\M{{\mathcal M}}

\def\diam{\mbox{\rm diam} }

\def\le{\leqslant}
\def\ge{\geqslant}

\def\F{\mathcal{F}}
\def\M{\mathcal{M}}

\newcommand{\vertiii}[1]{{\left\vert\kern-0.25ex\left\vert\kern-0.25ex\left\vert #1 
    \right\vert\kern-0.25ex\right\vert\kern-0.25ex\right\vert}}
\newcommand{\invertiii}[1]{{\vert\kern-0.25ex\vert\kern-0.25ex\vert #1 
    \vert\kern-0.25ex\vert\kern-0.25ex\vert}}

\begin{document}

\title{Non-compact spaces of invariant measures}
\date{\today}


\begin{thanks}
{We thank Raimundo Brice\~no for his comments and insights regarding Remark \ref{rai} and Mike Todd for many interesting observations. G.I.\ was partially supported  by Proyecto Fondecyt 1230100. A.V.\ was partially supported  by Proyecto Fondecyt Iniciaci\'on 11220409.}
\end{thanks}

\author[G.~Iommi]{Godofredo Iommi}
\address{Facultad de Matem\'aticas,
Pontificia Universidad Cat\'olica de Chile (UC), Avenida Vicu\~na Mackenna 4860, Santiago, Chile}
\email{\href{giommi@uc.cl}{giommi@uc.cl}}
\urladdr{\url{https://sites.google.com/view/godofredo-iommi/}}

 \author[A.~Velozo]{Anibal Velozo}  \address{Facultad de Matem\'aticas,
Pontificia Universidad Cat\'olica de Chile (UC), Avenida Vicu\~na Mackenna 4860, Santiago, Chile}
\email{\href{apvelozo@uc.cl}{apvelozo@uc.cl}}
\urladdr{\href{https://sites.google.com/view/apvelozo/inicio}{https://sites.google.com/view/apvelozo/inicio}}

\begin{abstract}
We study a compactification of the space of invariant probability measures for a transitive countable Markov shift. We prove that it is affine homeomorphic to the Poulsen simplex. Furthermore, we establish that, depending on a combinatorial property of the shift space, the compactification contains either a single new ergodic measure or a dense set of them. As an application of our results, we prove that the space of ergodic probability measures of a transitive countable Markov shift is homeomorphic to $\ell_2$, extending to the non-compact setting a known result for subshifts of finite type. Additionally, we explore implications for thermodynamic formalism, including a version of the dual variational principle for transitive countable Markov shifts with uniformly continuous potentials.
\end{abstract}

\maketitle

\section{Introduction}

A major breakthrough in the theory of dynamical systems was the realization that studying  invariant measures, instead of orbits, provided a significant gain in structure, enabling the proof of important results. The study of the topology and affine geometry of the space of invariant probability measures for a dynamical system has been a central theme in ergodic theory since its early stages. Some of the first results in this direction were proved in 1937 by Kryloff and Bogoliouboff \cite{kb}, spurred by the remarkable development of ergodic theory during that decade. Indeed, they showed that the space of invariant measures for a continuous map defined on a compact metric space is non-empty \cite[Theorem I, p.92]{kb}. Interestingly, to achieve this, they proved that the space of probability measures is compact with respect to the weak* topology \cite[Theorem I, p.69]{kb}. Moreover, they observed that the space of invariant measures is convex \cite[Corollaire, p.106]{kb}. Since ergodic measures play a fundamental role in the theory, understanding their inclusion in the space of invariant measures is particularly relevant. It was soon realized that ergodic measures are the extreme points of the convex set of invariant measures. Furthermore, in the 1960s, Oxtoby \cite{o} and Parthasarathy \cite{par} showed that, under mild assumptions, the set of ergodic measures form a $G_{\delta}-$set with respect to the weak* topology \cite[Theorem 2.1]{par}.

 Interest in this topic was further boosted by the work of Sigmund \cite{si} in the early 1970s. His research focused on Axiom A diffeomorphisms and subshifts of finite type. Among other results, he showed that the set of invariant probability measures supported on periodic orbits, referred to as \emph{periodic measures}, is dense in the space of invariant measures, and that the set of ergodic measures is residual (a result already noted by Parthasarathy \cite[Theorem 3.2]{par}). These findings, along with those in \cite{los}, establish that the space of invariant probability measures for a transitive subshift of finite type is affine homeomorphic to the Poulsen simplex. The Poulsen simplex is a distinguished infinite-dimensional Choquet simplex, characterized by the property that its set of extreme points is dense. More recently, Gelfert and Kwietniak \cite{gk} identified conditions on abstract dynamical systems that  guarantee the corresponding space of invariant measures forms a Poulsen simplex. Regarding other simplices, Downarowicz \cite{d} obtained a remarkable realization result in 1991, proving that for every Choquet simplex, there exists a minimal subshift whose space of invariant probability measures is affine homeomorphic to that simplex.

In this article, we study non-compact spaces of invariant measures. More precisely, we study the space of invariant probability measures of countable Markov shifts. These systems generalize subshifts of finite type by allowing the alphabet to contain countably many symbols. Countable Markov shifts are  dynamical systems defined on non-compact phase spaces that exhibit a great deal of complexity.




The study of ergodic properties of countable Markov shifts has been recently boosted by the construction of Markov partitions on countable alphabets for large classes of dynamical systems. For instance, non-uniformly hyperbolic diffeomorphisms defined on compact manifolds \cite{ov, sa4}, Sinai and Bunimovich billiards \cite{lm}, and the geodesic flow on non-compact hyperbolic surfaces \cite{gka}.  For a survey of this topic, see \cite{l}, and for striking applications of this method, refer to \cite{bcs}. 

\subsection{Compactification of the space of invariant probability measures} 

Compactifications of the phase space of countable Markov shifts have been studied. For instance, Gurevich \cite{gu2}, Walters \cite{Wal78}, Zargaryan \cite{Zar86},  Gurevich-Savchenko \cite{gs} and Iommi-Todd \cite{it2} considered  metric compactifications in order to study thermodynamic formalism in this setting. For locally compact countable Markov shifts, Fiebig and Fiebig \cite{ff1}  constructed compactifications that are larger than the one point compactification. 
In the same setting, Schwartz \cite{sh} extended the notion of  Martin boundary and obtained results related to the corresponding transfer operator. More generally, for systems defined on non-compact phase spaces, Handel and Kitchens \cite{hk} investigated various notions of entropy, considering all possible metric compactifications of the phase space. Our interest, in contrast with these works, is in the space of invariant measures. 


 %

In \cite{iv}, we introduced and studied a particular class of countable Markov shifts satisfying a combinatorial condition known as the $\mathcal{F}-$\emph{property} (see Section \ref{F} for details). Examples of systems that satisfy the $\mathcal{F}-$property include locally compact countable Markov shifts and countable Markov shifts with finite entropy. With the aid of a topology that captures the escape of mass phenomenon, the \emph{cylinder topology}, we proved that if $(\Sigma, \sigma)$ satisfies the $\F-$property, the space of invariant sub-probability measures, $\M_{\leq 1}(\Sigma,\sigma)$, is compact and affine homeomorphic to the Poulsen simplex (see \cite[Theorem 1.2]{iv}). As a consequence, we showed that the space of invariant probability measures is affine homeomorphic to the Poulsen simplex minus a vertex and its corresponding convex combinations. On the other hand, if $(\Sigma,\sigma)$ does not satisfy the $\F-$property, then $\M_{\leq 1}(\Sigma,\sigma)$ is non-compact (see \cite[Proposition 4.19]{iv}).  

In this work, we extend our study to the general case of transitive countable Markov shifts without any combinatorial assumptions. Our approach in this broader setting is different: we consider a metric compactification of the phase space, leading to a compact dynamical system $(\bar{\Sigma}, \bar{\sigma})$, similar to the one first studied by Zargaryan \cite{Zar86} (see also \cite{gs, it2}). In this compactification, a new symbol, denoted by $\infty$, is added to the alphabet; however, the resulting dynamical system is no longer a countable Markov shift. The space of invariant probability measures for the system $(\Sigma, \sigma)$ is denoted by $\M(\Sigma, \sigma)$ and the space of invariant probability measures for the compactified system $(\bar{\Sigma}, \bar{\sigma})$ is denoted by $\mathcal{M}(\bar{\Sigma}, \bar{\sigma})$. We establish the following result. 

\begin{theorem} \label{thm:po}
Let $(\Sigma,  \sigma)$ be a transitive countable Markov shift. Then, $\M(\bar{\Sigma},\bar{\sigma})$ is affine homeomorphic to the Poulsen simplex, and $\M(\Sigma,\sigma)$ is a dense subset of $\M(\bar{\Sigma},\bar{\sigma})$.
\end{theorem}

We provide a description of the ergodic measures added in the compactification. A striking dichotomy arises: either the compactification adds only a single new extreme point, corresponding to the atomic measure on the fixed point $\bar{\infty}=(\infty, \infty, \dots)$, that we denote by $\delta_{\bar{\infty}}$, or it produces a dense set of new extreme points. The first case characterizes the $\F-$property and can be used to obtain a new proof of the compactness of $\M_{\le 1}(\Sigma,\sigma)$ under this assumption. 

\begin{theorem} \label{thm:erg}
Let $(\Sigma,  \sigma)$ be a transitive countable Markov shift. Then:
\begin{enumerate}
\item \label{1}  The system $(\Sigma, \sigma)$  satisfies the $\F-$\emph{property} if and only if  the unique ergodic measure in $\M(\bar\Sigma,\bar\sigma) \setminus \M(\Sigma,{\sigma})$ is $\delta_{\bar{\infty}}$. 
 \item \label{2} The system $(\Sigma, \sigma)$  does not satisfy the $\F-$\emph{property} if and only if  the set of ergodic measures in $\M(\bar\Sigma,\bar\sigma) \setminus \M(\Sigma,\sigma)$ is dense in $\M(\bar\Sigma,\bar\sigma)$.
\end{enumerate}
\end{theorem}


A countable Markov shift that does not satisfy the $\mathcal{F}-$property admits sequences of invariant probability measures that converge, in the cylinder topology, to finitely additive measures that are not countably additive (see \cite[Proposition 4.19]{iv}). These limiting objects are identified with measures in $\M(\bar\Sigma, \bar\sigma)$ that lie outside the convex hull of $\M(\Sigma, \sigma) \cup \{\delta_{\bar\infty}\}$ (see Corollary \ref{nme}). To prove this, we relate the cylinder topology on the space of invariant sub-probability measures to the weak$^*$ topology on $\M(\bar\Sigma, \bar\sigma)$ (see Theorem \ref{topo:r}).

\subsection{\bf The space of ergodic measures} 
The Poulsen simplex plays a key role in the study of spaces of invariant measures. Originally, Poulsen \cite{pou} constructed a compact simplex within $\ell_2$ where the set of extreme points is dense. A few years later, Lindenstrauss, Olsen, and Sternfeld \cite{los} proved that any two metrizable simplices with a dense set of extreme points are affine homeomorphic. They proved several remarkable properties of the Poulsen simplex, for instance, that the Poulsen simplex is homogeneous, universal and that its set of extreme points is homeomorphic to $\ell_2$. As a consequence, the space of ergodic probability measures for subshifts of finite type is homeomorphic to $\ell_2$. In Section \ref{thespaceofmeasures}, we extend this result to countable Markov shifts.

\begin{theorem}\label{theo_l2} Let $(\Sigma,\sigma)$ be a transitive countable Markov shift. Then, the space of ergodic probability measures of $(\Sigma,\sigma)$ is homeomorphic to $\ell_2$.
\end{theorem}

It is a theorem of Anderson \cite{an} that $\ell_2$ is homeomorphic to $\R^\N$ with the product topology. Consequently, the space of ergodic invariant measures is always homeomorphic to a countable product of lines, regardless of whether the phase space is compact. Therefore, this result provides a topological characterization of the space of ergodic invariant measures. It also sheds light on the complicated nature of this space.

\subsection{Thermodynamic formalism}

The study of thermodynamic formalism in the context of (compact) subshifts of finite type is well established,  with a well-developed and largely complete theory. In contrast, despite significant progress in recent years, the theory for (non-compact) countable Markov shifts remains less developed. The lack of compactness in both the phase space and the space of invariant probability measures introduces significant difficulties, restricting the use of techniques that have proven effective in the compact setting. In this work, we present several applications of our compactification method within the framework of thermodynamic formalism for countable Markov shifts. These will allow for the application of methods and techniques used in the compact setting in this non-compact realm.

For a function (also called potential) $\phi:\Sigma\to\R$, we denote its pressure  by $P(\phi)$ (see Subsection \ref{tf} for details) and let $s_\infty(\phi)=\inf\{t\in[0,\infty): P(t\phi)<\infty\}$. We denote by $\text{UC}_d(\Sigma)$ the space of uniformly continuous real valued functions with respect to the metric $d$ defined in Section \ref{sec:pre} and by $C_b(\Sigma)$ the space of real valued continuous and bounded functions on $\Sigma$. In Section \ref{sec:dvp} we prove the following dual variational principle.

\begin{theorem}\label{theo:dual} Let $(\Sigma, \sigma)$ be a transitive countable Markov shift. Let $\phi\in\emph{UC}_d(\Sigma)$ be potential such that $s_\infty(\phi)<1$, $\sup \phi <\infty$ and $P(\phi)<\infty$. If $\mu \in \M(\Sigma,\sigma)$ satisfies that $\int \phi d\mu>-\infty$, then
\begin{equation*}
h(\mu) + \int \phi \, d \mu = \inf \left\{P(\phi+g) - 	\int g \, d \mu : g \in C_b(\Sigma)	\right\}.
\end{equation*}
\end{theorem}

Inspired by its applications in statistical mechanics, convex analysis has been extensively and effectively employed in the study of thermodynamic formalism for subshifts of finite type. The dual variational principle establishes that the subdifferentials of the pressure functional correspond to equilibrium measures, a fundamental property underlying many applications of convex analysis in this field 
(see, for instance, \cite{is,ph,bcm,w2}). 

In Section \ref{sec:equi} we prove that every ergodic measure with a finite entropy assumption is the unique equilibrium measure of a continuous function (see Theorem \ref{eem}) and that the set of measures that are not equilibrium measures for any continuous function form a dense set (Theorem \ref{no-eq}).

\section{Preliminaries}\label{sec:pre}

In this section we review known facts about countable Markov shifts and their spaces of invariant measures. 

\subsection{Countable Markov Shifts} \label{cms}
Let  $M$ be a $\N \times  \N$ infinite matrix  with entries $0$ or $1$. The symbolic space associated to $M$ is defined by
 \begin{equation*}
 \Sigma=\left\{ (x_1, x_2, \dots) \in \N^{\N}: M(x_i,x_{i+1})=1 \text{ for every } i \in \N \right\}.
\end{equation*} 
We endow $\N$ with the discrete topology and $\N^{\N}$ with the product topology. On $\Sigma$ we consider the induced topology.  Note that, in general, this is a non-compact space.   The \emph{shift map} $\sigma:\Sigma \to \Sigma$ is defined by $\sigma(x)=(x_2,x_3,\ldots)$, where $x=(x_1, x_2, \dots ) \in \Sigma$. The dynamical system $(\Sigma,\sigma)$ is called  a  \emph{countable Markov shift}.

An \emph{admissible word} of length $N$ is a string ${\bf w} =a_1a_2\ldots a_{N}$ of letters in $\N$ such that $M(a_i,a_{i+1})=1$ for every $i\in\{1,\ldots,N-1\}$. A \emph{cylinder} of length $N$ is a set of the form 
\begin{equation*}
[a_1,\ldots,a_{N}]= \left\{ (x_1,x_2,\ldots)\in \Sigma :  x_i=a_i  \text{ for } 1 \le i \le N \right\}.
\end{equation*} 
Cylinder sets form a basis for the topology on $\Sigma$. 
We stress that if ${\bf w} =a_1a_2\ldots a_{N}$ is not an admissible word then the associated cylinder is empty.

A countable Markov shift $(\Sigma, \sigma)$ is associated with a directed graph $G = (V, E)$, where points in $\Sigma$ correspond to paths over the graph. More precisely, the set of vertices $V$ is identified with $\mathbb{N}$, and there is a directed edge $i \to j$ in $E$ if and only if $M(i, j) = 1$.

The space $\Sigma$ is metrizable and there are many different metrics that generate the topology. We will be interested in two particular metrics on this set. The first is the metric $d$ which is defined by: $d(x, y) = 0$ if $x = y$, and for $x \neq y$, let $d(x, y) = \frac{1}{2^{m-1}}$, where $m = \min \{i \in \mathbb{N} : x_i \neq y_i\}$. Let $\rho$ be the following metric in $\N$:
\begin{equation*}
\rho(a, b)=
\begin{cases}
\left|\frac1a-\frac1b\right| & \text{ if } a, b\in \N \text{ and } a\neq b;\\
0 & \text{ if } a=b.
\end{cases}
\end{equation*}
The second metric is $d_{\rho}$ and it is  defined for $ x=(x_1, x_2, \ldots),  y= (y_1, y_2, \ldots)\in \Sigma,$ by
\begin{equation}\label{eq_metric}
d_{\rho}( x,  y)=\sum_{n\in\N}\frac{1}{2^n}\rho(x_n, y_n).
\end{equation}
Similarly, any finite diameter metric on $\N$ can be used to induce a metric on $\Sigma$.

Several spaces of functions will be considered in this article. Let $C(\Sigma)$ denote the space of real-valued continuous functions on $\Sigma$, and let $C_b(\Sigma)$ be the subspace of bounded functions. Denote by $\text{UC}_d(\Sigma)$ and $\text{UC}_\rho(\Sigma)$  the spaces of uniformly continuous functions with respect to the metrics $d$ and    $d_\rho$, respectively.  Observe that if $d(x, y) = \frac{1}{2^m}$, then $d_\rho(x, y) \leq \frac{1}{2^m}$, thus $d_\rho(x, y) \leq d(x, y).$ 
This implies that   $\text{UC}_\rho(\Sigma) \subseteq \text{UC}_d(\Sigma)$.

For a function $\phi:\Sigma\to \R$ we define its $n$th-variation by $$\text{var}_n(\phi)=\sup\{|\phi(x)-\phi(y)|: d(x,y)\le 2^{-n}\}.$$
Note that $\phi\in \text{UC}_d(\Sigma)$ if and only if $\lim_{n\to \infty}\text{var}_n(\phi)=0$. We say that $\phi$ is \emph{weakly H\"older} if there exist $C>0$ and $\lambda\in (0,1)$ such that $\text{var}_n(\phi)\le C\lambda^n$ for all $n\ge 2$.
      
The system  $(\Sigma,\sigma)$ is \emph{topologically transitive} (or transitive) if for every pair $x,y\in \N$, there exists an admissible word connecting $x$ with $y$. We say that $(\Sigma, \sigma)$ is \emph{topologically mixing} (or mixing) if for every pair $x,y\in \N$, there exists $N \in \N$ such that for every $n \ge N$, there is an admissible word of length $n$ connecting $x$ with $y$. These definitions coincide with the classical notions of topological transitivity and mixing for topological dynamical systems. 

\begin{remark} Throughout this article, we assume that $(\Sigma, \sigma)$ is a transitive countable Markov shift and, in particular, that $\Sigma$ is a non-compact space.
\end{remark}


\subsection{The weak$^*$  and the cylinder topologies} \label{topologies}
Denote by \(\mathcal{M}(\Sigma, \sigma)\) the space of \(\sigma\)-invariant Borel probability measures on \(\Sigma\), and by \(\mathcal{M}_{\leq 1}(\Sigma, \sigma)\) the space of \(\sigma\)-invariant Borel sub-probability measures on \(\Sigma\). Every measure in \(\mathcal{M}_{\leq 1}(\Sigma, \sigma)\), except for the zero measure, can be uniquely written as \(\lambda \mu\), where \(\lambda \in (0,1]\) and \(\mu \in \mathcal{M}(\Sigma, \sigma)\).  We denote by \(\mathcal{M}^e(\Sigma, \sigma)\) the set of ergodic probability measures of \((\Sigma, \sigma)\).   In this work, we consider two distinct notions of convergence for measures.

\begin{definition}
A sequence of measures $(\mu_n)_n$ in  $\M(\Sigma,{\sigma})$ converges weak$^*$ to a measure $\mu$  if for every $f \in C_b(\Sigma)$ we have 
\begin{equation*}
\lim_{n \to \infty} \int f \, d \mu_n = \int f \, d \mu.
\end{equation*}
\end{definition}


This notion of convergence defines a metrizable topology on $\mathcal{M}(\Sigma, \sigma)$, known as the weak$^*$ topology. Since the constant function equal to one is an element of $C_b(\Sigma)$, it follows that if a sequence in $\mathcal{M}(\Sigma, \sigma)$ converges to a measure $\mu$ in the weak$^*$ topology, then $\mu \in \mathcal{M}(\Sigma, \sigma)$.  When $\Sigma$ is compact, the space $\mathcal{M}(\Sigma, \sigma)$ is also compact (see \cite[Theorem 6.10]{w}). However, if $(\Sigma, \sigma)$ is non-compact and transitive, then $\mathcal{M}(\Sigma, \sigma)$ is non-compact. On the other hand, $\mathcal{M}(\Sigma, \sigma)$ is a convex set whose extreme points are the ergodic measures (see \cite[Theorem 6.10]{w}). 

The following notion of convergence of measures was considered and studied in \cite{iv}.

\begin{definition}
Let $(\mu_n)_n$ and $\mu$ be measures in $\M_{\le 1}(\Sigma,\sigma)$. We say that $(\mu_n)_n$ converges on {cylinders}  to $\mu$ if for every cylinder $C \subseteq \Sigma$ we have
\begin{equation*}
\lim_{n \to \infty}  \mu_n(C)=  \mu(C).
\end{equation*}
\end{definition}

This notion of convergence defines a metrizable topology on $\mathcal{M}_{\leq 1}(\Sigma,\sigma)$ called the cylinder topology. A key advantage of this topology is that it captures the escape of mass (see \cite{itv} for applications). More precisely, a sequence of invariant probability measures can converge in the cylinder topology to a sub-probability measure. Furthermore,  when the space $\Sigma$ is locally compact, the cylinder topology coincides with the {vague} topology (see \cite[Section 3 and Lemma 3.18]{iv} for precise definitions and details). 

\begin{remark} If $\mu \in \mathcal{M}(\Sigma, \sigma)$ and $(\mu_n)_n$ is a sequence of measures in $\mathcal{M}(\Sigma, \sigma)$, then $(\mu_n)_n$ converges to $\mu$ in the weak$^*$ topology if and only if it converges to $\mu$ on cylinders (see \cite[Lemma 3.17]{iv}). In other words, when there is no escape of mass, the cylinder and the weak$^*$ topologies coincide. \end{remark}

\subsection{Thermodynamic formalism} \label{tf}
In this section, we briefly review the basic properties of the thermodynamic formalism in the non-compact setting under consideration. The \emph{pressure} of a function (or, equivalently, a potential) $\phi \in C(\Sigma)$ is defined by
\begin{equation*}
P(\phi)= \sup \left\{h(\mu) + \int  \phi \, d \mu :\mu \in \M(\Sigma,\sigma) \text{ such that } \int \phi d \mu> -\infty \right\},
\end{equation*}
where $h(\mu)$ denotes the entropy of the measure $\mu$ (see \cite[Chapter 4]{w}). For a function $\phi\in C(\Sigma)$, we define $$\M_\phi(\Sigma,\sigma)=\left\{\mu \in \M(\Sigma,\sigma) : \int \phi d \mu> -\infty \right\}.$$

A measure $\mu \in \mathcal{M}_\phi(\Sigma, \sigma)$ that attains the supremum in the definition of the pressure of $\phi$ is called an \emph{equilibrium measure} for $\phi$. A fundamental result in this context is that if  $(\Sigma,\sigma)$ is a transitive countable Markov shift and $\phi$ has summable variations, then $\phi$ has at most one equilibrium state (see \cite[Theorem 1.1]{bs}).

Analogous to the definition of topological pressure, we introduce a notion of pressure that quantifies the concentration of free energy in the non-compact regions of the phase space. The \emph{pressure at infinity} of a function \( \phi \in C(\Sigma) \) is defined by  
\[
P_\infty(\phi) = \sup_{(\mu_n)_n \to 0} \limsup_{n \to \infty} \left( h({\mu_n}) + \int \phi \, d\mu_n \right),
\]  
where the supremum is taken over all sequences \( (\mu_n)_n \) in \( \M_\phi(\Sigma, \sigma) \)  which converge to the zero measure in the cylinder topology. If no such sequence exists, we set \( P_\infty(\phi) = -\infty \). 

For a function $\phi \in C(\Sigma)$, we define $$s_\infty(\phi )=\inf\{t\in [0,\infty):P(t\phi )<\infty\}.$$
Note that if $\phi$ has finite pressure, then $s_\infty(\phi)\in [0,1]$. 


For our applications, we will need the  following result which relates the escape of mass phenomenon with the upper semi-continuity of the pressure map (see \cite[Theorem 1.3]{v}). 

\begin{theorem} \label{thm:usc}  Let $(\Sigma,\sigma)$ be a transitive countable Markov shift and $\phi\in \emph{UC}_d(\Sigma)$ a potential such that $\emph{var}_2(\phi)<\infty$ and $P(\phi)<\infty$.  Let $(\mu_n)_n$ be a sequence  in $\M_\phi(\Sigma,\sigma)$ which converges on cylinders to $\lambda\mu$, where $\lambda\in [0,1]$ and $\mu\in\M_\phi(\Sigma,\sigma)$. Suppose that $\liminf_{n\to\infty}\int \phi  d\mu_n>-\infty$, or that $s_\infty(\phi)<1$. Then,
$$\limsup_{n\to\infty}\bigg(h({\mu_n})+\int \phi d\mu_n\bigg)\le \lambda\bigg(h(\mu)+\int \phi d\mu\bigg)+(1-\lambda)P_\infty(\phi).$$
Moreover, the inequality is sharp.
\end{theorem}

\subsection{The $\F-$property.}\label{F}
In this section, we review some essential facts about the $\mathcal{F}$-property, a combinatorial condition introduced and studied in \cite{iv}.

\begin{definition} \label{def:F}  
A countable Markov shift is said to satisfy the $\F-$\emph{property} if for every element of the alphabet $a$ and natural number $n$ there are only finitely many admissible words of length $n$ that begin and end with $a$. Equivalently, if for every element of the alphabet $a$ and natural number $n$ there are only finitely many periodic orbits of length $n$ that intersect $[a]$.
\end{definition}

Every countable Markov shift with finite entropy and every locally compact countable Markov shift satisfy the $\mathcal{F}-$property. The full shift on a countable alphabet, however, does not satisfy the $\mathcal{F}$-property. On the other hand, there are countable Markov shifts with infinite entropy that are not locally compact but still satisfy the $\mathcal{F}$-property.

\begin{example}[Loop systems]\label{bouquet}
Consider a countable Markov shift defined by a graph consisting of simple loops, all based at a common vertex and otherwise disjoint. This structure is called a loop system. A loop system is characterized by a sequence $(a_n)_{n \in \mathbb{N}}$ in $\mathbb{N} \cup \{\infty\}$, where $a_n$ denotes the number of simple loops of length $n$. A loop system:
\begin{enumerate}
    \item satisfies the $\mathcal{F}$-property if and only if $a_n \in \mathbb{N}$ for every $n \in \mathbb{N}$
    \item has finite entropy if and only if $a_n \in \mathbb{N}$ for all $n \in \mathbb{N}$ and $\limsup_{n \to \infty} \frac{1}{n} \log a_n < \infty$
    \item is locally compact if and only if $\sum_{n \in \mathbb{N}} a_n$ is finite.
\end{enumerate}
In particular, there exist loop systems that satisfy the $\mathcal{F}-$property but are neither locally compact nor of finite entropy.
\end{example}

A complete description of the topology and convex structure of $\M_{\leq 1}(\Sigma, \sigma)$ is available when $(\Sigma, \sigma)$ satisfies the $\mathcal{F}$-property. It has been shown that this space is affine homeomorphic to a special simplex, which we now proceed to define.

\begin{definition} \label{poulsen}
A metrizable convex compact Choquet simplex $K$ with at least two points  is a \emph{Poulsen simplex} if its extreme points are dense in $K$.
\end{definition} 

The first example of such a simplex was constructed by Poulsen in 1961 \cite{pou}. In 1978, Lindenstrauss, Olsen, and Sternfeld proved that the Poulsen simplex is unique up to affine homeomorphism (see \cite[Theorem 2.3]{los}). They also proved several striking properties of this simplex. For example, they showed that the set of extreme points forms a dense $G_{\delta}$-set that is path-connected (see \cite[Section 3]{los}). Consequently, the Poulsen simplex has the remarkable property that it is a convex set in which, from a topological perspective, most points are extreme points. 

It is well known that the space of invariant probability measures is a Choquet simplex, where its extreme points are the ergodic measures (see \cite[Section 6.2]{w}). In the early 1970s, Sigmund \cite{si} proved that for Axiom A diffeomorphisms and for sub-shifts of finite type, the set of ergodic measures is dense. This was also observed by Ruelle \cite[Section 6]{r} (see also the work of Olsen \cite[pp. 45-46]{ol}).  By \cite[Theorem 2.3]{los}, this shows that the space of invariant probability measures is a Poulsen simplex.

The first part of the following result follows from \cite[Theorem 4.18, Proposition 4.19]{iv}, while the second part is from \cite[Theorem 5.1]{iv}.

\begin{theorem} \label{sub:comp}
Let $(\Sigma, \sigma)$ be a transitive countable Markov shift. Then, $\M_{\leq 1}(\Sigma,\sigma)$ is compact with respect to the cylinder topology if and only if $(\Sigma, \sigma)$ satisfies the $\mathcal{F}-$property. Moreover, if $(\Sigma, \sigma)$ satisfies the $\mathcal{F}-$property, then $\M_{\leq 1}(\Sigma, \sigma)$ is affine homeomorphic to the Poulsen simplex, and $\M(\Sigma, \sigma)$ is a dense subset.
\end{theorem}

In particular, if $(\Sigma,\sigma)$ has the $\F-$property, then $\M_{\le1}(\Sigma,\sigma)$ is a compactification of $\M(\Sigma,\sigma)$. In this article, we study a compactification of the space of invariant probability measures that  is applicable to all countable Markov shifts. This compactification is affine homeomorphic to $\M_{\leq 1}(\Sigma,\sigma)$ when $(\Sigma,\sigma)$ satisfies the $\mathcal{F}-$property (see Remark \ref{rem:Fequiv}), recovering Theorem \ref{sub:comp} as a particular case.

\subsection{Finite uniform Rome}

We now introduce another combinatorial property of $(\Sigma,\sigma)$, originally studied by Cyr \cite{c} in the context of the thermodynamic formalism for countable Markov shifts. A similar notion was introduced in the context of interval maps in \cite{blo} (see also \cite{bt}).

\begin{definition} \label{rome}
Let $G = (V, E)$ be the directed graph associated with a countable Markov shift $(\Sigma, \sigma)$. A subset $F \subseteq V$ is called a \emph{uniform Rome} if there exists $N \in \mathbb{N}$ such that $V \setminus F$ contains no paths in $G$ of length greater than $N$. A \emph{finite uniform Rome} is a uniform Rome where $F$ is finite.
\end{definition}

The following result summarizes the properties of systems with a finite uniform Rome. We are particularly interested in the relationship between systems with a finite uniform Rome and the existence of sequences of invariant probability measures converging to the zero measure in the cylinder topology. For completeness, we also present results concerning the dynamical properties of a class of regular potentials. 

Let $\Phi$ denote the collection of all weakly H\"older functions with finite pressure. Functions in $\Phi$ with good thermodynamic properties--those with stable equilibrium measures--are called \emph{strongly positive recurrent} (SPR). Functions with bad thermodynamic properties--those with no conformal measures--are called \emph{transient}. For precise definitions, see \cite{s2}. Note that, in general, there exist potentials in $\Phi$ that are neither transient nor SPR.

\begin{theorem}\label{cyr} Let $(\Sigma,\sigma)$ be a transitive countable Markov shift and $G$ its corresponding directed graph. Then, the following statements are equivalent:
\begin{enumerate}
\item\label{a} The graph $G$ has a finite uniform Rome.
\item\label{b} The set of transient potentials in $\Phi$  is empty.
\item\label{c} Every potential in $\Phi$ is SPR. 
\item\label{d} There is no sequence in $\M(\Sigma,\sigma)$ which converges on cylinders to the zero measure. 
\end{enumerate}
\end{theorem}

The equivalence of statements $(a)$ and $(b)$ was established by Cyr in \cite[Theorem 2.1]{c}. The equivalence between $(b)$ and $(c)$ follows from \cite[Theorem 2.3]{cs}, while the equivalence of $(c)$ and $(d)$ is proved in \cite[Theorem 4.3]{v}.

\begin{example}\label{exampleloop}
Consider the loop system defined by $a_1 = 1$ (a single loop of length one), $a_2 = \infty$ (infinitely many loops of length two), and no additional simple loops (see Example \ref{bouquet}). This system admits a finite uniform Rome. Indeed, let $F$ consist of the base vertex, and set $N = 2$. Observe that for any periodic measure, the mass of the cylinder associated with the base vertex is greater than $\frac{1}{2}$. Consequently, the same holds for any invariant probability measure, implying that no sequence of invariant probability measures can converge on cylinders to the zero measure. This example is discussed in \cite[Example A.2]{c}, and similar constructions appear in \cite{rue}.
\end{example}

\begin{remark}
If $(\Sigma, \sigma)$ has the $\mathcal{F}$-property, then there exist sequences in $\M(\Sigma, \sigma)$ that converge on cylinders to the zero measure (see Theorem \ref{sub:comp}). On the other hand, if $(\Sigma, \sigma)$ has a finite uniform Rome, then no sequence in $\M(\Sigma, \sigma)$ converges on cylinders to the zero measure (see Theorem \ref{cyr}). The class of systems with the $\mathcal{F}$-property is disjoint from the class of systems with a finite uniform Rome.
\end{remark}


\begin{example} (Example  with no finite uniform Rome and no $\F-$property.)
The simplest example with these properties is the full-shift. For another example,
let $(\Sigma,\sigma)$ be a countable Markov shift defined by a graph consisting of infinitely many loops of length $2$ based at the vertex labeled by $0$ together with a random walk sub-graph based at $0$. That is, a sub-graph with vertices labeled by $\Z$, for which the vertex $n$ is connected to the vertices labeled by $n-1$ and $n+1$. Then,  $(\Sigma,\sigma)$ does not satisfy the $\F-$property since the vertex $0$ has infinitely many loops of length $2$ and it does not have a finite uniform Rome since it contains a ray escaping to infinity (see \cite[Lemma 5.2]{c}).
\end{example}


\section{Metric compactification of $\Sigma$}

In this section, we propose a completion approach to studying the space of invariant probability measures for a transitive countable Markov shift. 

\subsection{A compactification of $\Sigma$}\label{comp} 
Let $\rho$ denote the metric on $\N$ defined in subsection \ref{cms}. Observe that $\rho$ is totally bounded, and the completion of $\N$ with respect to $\rho$ is a compact metric space that can be identified with $\N \cup \{\infty\}$. The metric $\rho$ extends to $\N \cup \{\infty\}$, and we denote the extension by $\bar{\rho}$. Note that $\bar{\rho}(n, \infty) = \frac{1}{n}$ for every $n \in \N$.

Let $(\Sigma, \sigma)$ be a countable Markov shift. Since the metric $d_{\rho}$ is totally bounded, its metric completion, denoted by $\bar{\Sigma}$, is a compact metric space. In particular, $\bar\Sigma$ is a compactification of $\Sigma$. The compactification can be described more explicitly. Consider the full shift $\Sigma_{full}:=\N^\N$ equipped with the metric $d_\rho$. Its completion, $\bar\Sigma_{full}$, can be identified with $(\N\cup\{\infty\})^\N$ equipped with the extended metric $\bar{d}_\rho$, where 
$$\bar{d}_\rho(x,y)=\sum_{n\in\N}\frac{\bar\rho(x_i,y_i)}{2^n},$$
for $x=(x_1,x_2,\ldots)$ and $y=(y_1,y_2,\ldots)$ in $(\N\cup\{\infty\})^\N$. To see this, note that $\Sigma_{full}\subseteq (\N\cup\{\infty\})^\N$ is dense, $\bar{d}_\rho$ restricts to $d_\rho$ on $\Sigma_{full}$, and $((\N\cup\{\infty\})^\N,\bar{d}_\rho)$ is a compact metric space. Since $\Sigma\subseteq \Sigma_{full}$, it follows that $\bar\Sigma$ is the closure of $\Sigma$ in $((\N\cup\{\infty\})^\N,\bar{d}_\rho)$. Consequently, every point in $\bar\Sigma$ can be represented as an infinite sequence of elements in $\N \cup \{\infty\}$. Note that a sequence of points $x^n=(x_1^n,x_2^n,\ldots)\in \bar\Sigma$ converges to $y=(y_1,y_2,\ldots)\in\bar\Sigma$ if and only if $\lim_{n\to\infty}x_i^n=y_i$ for all $i\in\N$.

The shift map $\sigma: \Sigma \to \Sigma$ is uniformly continuous with respect to $d_\rho$, and we denote its extension to $\bar{\Sigma}$ by $\bar{\sigma}$. We will study the continuous dynamical system $(\bar{\Sigma}, \bar{\sigma})$, which serves as a compactification of $(\Sigma, \sigma)$. 

Let $\mathcal{M}(\bar{\Sigma},\bar{\sigma})$ denote the space of $\bar{\sigma}$-invariant probability measures on $\bar{\Sigma}$, which endowed with the weak$^*$ topology is a compact space.

Let $\text{UC}_{\rho,b}(\Sigma) =\text{UC}_\rho(\Sigma) \cap C_b(\Sigma)$ be the space of bounded, uniformly continuous functions on $\Sigma$ with respect to $d_\rho$. Each function in $\text{UC}_{\rho,b}(\Sigma)$ can be uniquely extended to a continuous function on $\bar\Sigma$. Thus, there is a one-to-one correspondence between $\text{UC}_{\rho,b}(\Sigma)$ and $\emph{C}(\bar\Sigma)$ via continuous extension and restriction.

   
\begin{remark} \label{rem:topextend}
The space $\emph{UC}_{\rho,b}(\Sigma)$ is a convergence defining class for the weak* topology (see  \cite[8.3.1 Remark]{bg}). More precisely, a sequence of measures $(\mu_n)_n$ converges to $\mu$ in the weak$^*$ topology if and only if $\lim_{n\to\infty}\int \phi d\mu_n=\int \phi d\mu$, for every $\phi\in \emph{UC}_{\rho,b}(\Sigma)$. Therefore, if $(\mu_n)_n$ and $\mu$ are measures in $\M(\bar\Sigma,\bar\sigma)$ such that $\mu_n(\Sigma)=\mu(\Sigma)=1$ for every $n\in\N$, then $(\mu_n)_n$ converges to $\mu$ in the weak$^*$ topology of $\M(\Sigma,{\sigma})$  if and only if  $(\mu_n)_n$ converges to $\mu$ in the weak$^*$ topology of   $\M(\bar\Sigma,{\bar\sigma})$. In other words, the inclusion $\M(\Sigma,\sigma)\subseteq \M(\bar\Sigma,\bar\sigma)$ is compatible with respect to the weak$^*$ topologies.
 \end{remark}

We say that a word ${\bf w} = w_1 \ldots w_n$ in the alphabet $\N \cup \{\infty\}$ is \emph{$\bar{\Sigma}$-admissible} if there exists a point $x = (x_1, x_2, \ldots) \in \bar{\Sigma}$ and an index $k \in \N\cup\{0\}$ such that $x_{k+i} = w_i$ for all $1 \leq i \leq n$. Let $s({\bf w})$ and $e({\bf w})$ denote the starting and ending symbols of the word ${\bf w}$, respectively. Observe that ${\bf w}=w_1\ldots w_n$ is $\bar\Sigma$-admissible if and only if there are admissible words ${\bf w^k}=w_1^k\ldots w_n^k$ in $\Sigma$ such that $\lim_{k\to\infty} w_i^k=w_i$ for all $1\le i\le n$. Combining this observation with the fact that the concatenation of two admissible words ${\bf w_1}$ and ${\bf w_2}$ in $\Sigma$ is admissible if and only if $M(e({\bf w_1}), s({\bf w_2})) = 1$, where $M$ is the defining matrix of $\Sigma$, we obtain the following result.

\begin{lemma}\label{lem_unir_adm} Suppose that ${\bf w_1}$ and ${\bf w_2}$ are $\bar{\Sigma}$-admissible words, where $e({\bf w_1}), s({\bf w_2}) \in \N$, and $M(e({\bf w_1}), s({\bf w_2})) = 1$. Then, the concatenated word ${\bf w_1w_2}$ is $\bar{\Sigma}$-admissible. Moreover, if $({\bf w_i})_{i \in \N}$ is a sequence of $\bar{\Sigma}$-admissible words such that $e({\bf w_i}), s({\bf w_{i+1}}) \in \N$ for all $i \in \N$, and $M(e({\bf w_i}), s({\bf w_{i+1}})) = 1$, then the point associated with the infinite concatenation ${\bf w_1w_2\ldots}$ belongs to $\bar{\Sigma}$.
\end{lemma}

A cylinder in $\bar\Sigma$ (also referred to as $\bar\Sigma$-cylinder) of length $N$ is a set of the form
\begin{equation*}
[a_1,\ldots,a_{N}]_{\bar\Sigma}:= \left\{ x=(x_1,x_2,\ldots)\in \bar\Sigma :  x_i=a_i  \text{ for } 1 \le i \le N \right\}.
\end{equation*} 
Note that a $\bar\Sigma$-cylinder is non-empty if and only if the associated word is $\bar\Sigma$-admissible. 

\begin{remark}\label{rem:clopen}
For $a \in \N$, note that $[a]_{\bar{\Sigma}} = \{x \in \bar{\Sigma} : \bar{d}_\rho(x, x_a) < 1/(2a(a+1))\}$, where $x_a$ is any point in $[a]_{\bar{\Sigma}}$. Furthermore, if $x \notin [a]_{\bar{\Sigma}}$, then $\bar{d}_\rho(x, [a]_{\bar{\Sigma}}) \geq 1/(2a(a+1))$. It follows that $[a]_{\bar{\Sigma}}$ is both open and closed in $\bar{\Sigma}$. Similarly, if ${\bf w} = a_1 \ldots a_n$ is an admissible word in $\Sigma$, then $[a_1, \ldots, a_n]_{\bar{\Sigma}}$ is also open and closed in $\bar{\Sigma}$. Equivalently, the characteristic function $\chi_{[a_1, \ldots, a_n]_{\bar{\Sigma}}}$ of $[a_1, \ldots, a_n]_{\bar{\Sigma}}$ is continuous on $\bar{\Sigma}$.
\end{remark}

\begin{remark}\label{rem:1} Set $A=\bigcup_{i\in\N}[i]_{\bar\Sigma}$. If $\mu\in \M(\bar\Sigma,\bar\sigma)$ satisfies $\mu([\infty]_{\bar\Sigma})=0$, then $\mu(A)=1$, and therefore $\mu(\Sigma)=\mu(\bigcap_{i=0}^\infty \bar\sigma^{-i}A)=1$. In particular, we can identify $\M(\Sigma, \sigma)$ with the measures in $\M(\bar\Sigma, \bar\sigma)$ that assign zero mass to $[\infty]_{\bar\Sigma}$, that is,
$$\M(\bar\Sigma,\bar\sigma)\setminus\M(\Sigma,\sigma)=\{\mu\in\M(\bar\Sigma,\bar\sigma): \mu([\infty]_{\bar\Sigma})\in (0,1]\}.$$
Note that if $\mu([\infty]_{\bar\Sigma})=1$, then $\mu(\{\bar\infty\})=\mu(\bigcap_{i=0}^\infty \bar\sigma^{-i}[\infty]_{\bar\Sigma})=1$, or equivalently, $\mu=\delta_{\bar\infty}$. 
\end{remark}

\noindent

\subsection{Topology on $\M(\bar\Sigma,{\bar\sigma})$} \label{rom-inf}
In this subsection we relate the convergence on cylinders and the escape of mass phenomenon in $\M(\Sigma,\sigma)$ with the weak* convergence in $\M(\bar\Sigma,\bar\sigma)$. 


\begin{theorem} \label{topo:r}
Let $(\Sigma, \sigma)$ be a transitive countable Markov shift. Let $(\mu_n)_n$ and $\mu$ be measures in $\mathcal{M}(\Sigma, \sigma)$. Then:
\begin{enumerate}
\item  \label{cyl_inf}
If the sequence $(\mu_n)_n$ converges on cylinders to $\lambda \mu$ for some $\lambda \in [0,1)$, then $(\mu_n)_n$ converges in the weak$^*$ topology of $\M(\bar\Sigma,{\bar\sigma})$ to the measure $\lambda \mu + (1- \lambda) \delta_{\bar\infty}$.

\item \label{rome_cyl} If  $\bar\infty\in\bar\Sigma$ and  $(\mu_n)_n$  converges in the weak$^*$ topology of $\M(\bar\Sigma,\bar\sigma)$ to the measure
$\lambda \mu + (1- \lambda) \delta_{\bar\infty}$, then $(\mu_n)_n$ converges on cylinders to $\lambda \mu$.
\end{enumerate}
\end{theorem}

\begin{proof}[Proof of Theorem \ref{topo:r} (\ref{cyl_inf})]
Since $\mathcal{M}(\bar\Sigma, \bar{\sigma})$ is compact with respect to the weak$^*$ topology, we can assume that $(\mu_n)_n$ converges to some measure $\nu \in \mathcal{M}(\bar\Sigma, \bar{\sigma})$, possibly after passing to a subsequence. By the ergodic decomposition, there exist $\alpha \in [0,1]$, $\mu_1 \in \mathcal{M}(\Sigma, \sigma)$, and $\mu_2 \in \mathcal{M}(\bar\Sigma, \bar{\sigma}) \setminus \mathcal{M}(\Sigma, \sigma)$, such that each ergodic component of $\mu_2$ is in $\mathcal{M}(\bar\Sigma, \bar{\sigma}) \setminus \mathcal{M}(\Sigma)$ (see Remark \ref{rem:1}), and we have the decomposition
\begin{equation*}
\nu= \alpha \mu_1 + (1-\alpha) \mu_2.
\end{equation*}
We will prove that $\mu_2 = \delta_{\bar\infty}$ and that $\alpha = \lambda$. Note that $\alpha < 1$; otherwise, the sequence would converge to $\mu_1$, which is a probability measure on $\Sigma$ (see Remark \ref{rem:topextend}).

Let $C=[a_1, \ldots, a_n]$ be cylinder in $\Sigma$, and let $C_{\bar\Sigma}=[a_1, \dots , a_n]_{\bar\Sigma}$ be the associated cylinder in $\bar\Sigma$. For $i\in\N\cup\{\infty\}$, we define $[C,i]=[a_1,\ldots,a_n,i]$. Since $\mu_n \in \M(\Sigma,{\sigma})$ and $(\mu_n)_n$ converges on cylinders to $\lambda \mu$, we have
\begin{equation*}
\lim_{n \to \infty}\mu_n(C_{\bar\Sigma})=\lim_{n \to \infty}\mu_n(C_{\bar\Sigma} \cap \Sigma) = \lambda \mu(C).
\end{equation*}
Since the  boundary of $C_{\bar\Sigma}$ is empty (see Remark \ref{rem:clopen}) and $(\mu_n)_n$ converges weakly to $\nu$, by Portmanteau's theorem, we  obtain 
\begin{equation*}
\lim_{n \to \infty} \mu_n(C_{\bar\Sigma})=\nu (C_{\bar\Sigma}).
\end{equation*}
Thus, we have
\begin{equation} \label{ec:1}
\lambda \mu(C)=\nu (C_{\bar\Sigma})=\alpha \mu_1(C) + (1-\alpha) \mu_2(C_{\bar\Sigma}).
\end{equation}
From this, we deduce
 \begin{align*} 
 \lambda \mu\big(\bigcup_{i\in\N} [C,i]\big)=  \alpha \mu_1\big(\bigcup_{i\in\N} [C,i]\big) +(1-\alpha) \mu_2\big( \bigcup_{i\in\N} [C,i]_{\bar\Sigma}\big).
  \end{align*}
Note that $C_{\bar\Sigma}=\bigcup_{i\in\N}[C,i]_{\bar\Sigma} \cup  [C, \infty]_{\bar\Sigma}.$ Therefore, we have
 \begin{equation*} 
  \lambda \mu(C) = \alpha \mu_1(C) +(1-\alpha) (\mu_2( C_{\bar\Sigma} ) - \mu_2( [C, \infty]_{\bar\Sigma}) ). \end{equation*}
Combining with equation \eqref{ec:1}, we obtain 
\begin{equation*}
\mu_2( [C, \infty]_{\bar\Sigma})=0.
\end{equation*}
Since this holds for any cylinder $C$ in $\Sigma$, we conclude that
\begin{equation*}
\mu_2([\infty]_{\bar\Sigma}) = \mu_2 (\bar\sigma^{-1}[\infty]_{\bar\Sigma})= \mu_2\left(\bigcup_{a \in \N} [a ,\infty]_{\bar\Sigma} \cup	[\infty ,\infty]_{\bar\Sigma}\right) =\mu_2([\infty ,\infty]_{\bar\Sigma}).
\end{equation*}
By induction, we obtain ${\mu}_2([\infty]_{\bar\Sigma})={\mu}_2([\infty,\ldots,\infty]_{\bar\Sigma})$. Thus, we conclude that $$\mu_2([\infty]_{\bar\Sigma})= \mu_2(\{\bar{\infty}\}).$$
Since $\mu_2\in \M(\bar\Sigma,{\bar\sigma})\setminus \M(\Sigma,{\sigma})$, it follows that $\mu_2([\infty])\in (0,1]$ (see Remark \ref{rem:1}). If $\mu_2([\infty])<1$, then $\mu_2=(1-\mu_2([\infty]))\eta+\mu_2([\infty])\delta_{\bar\infty}$, where $\eta\in \M(\bar\Sigma,\bar\sigma)$ and $\eta([\infty])=0$. This contradicts  the fact that ergodic components of $\mu_2$ are in $\M(\bar\Sigma,{\bar\sigma})\setminus \M(\Sigma,{\sigma})$. Therefore, we conclude that  $\mu_2= \delta_{\bar\infty}$, and hence, $\nu= \alpha \mu_1 +(1-\alpha)\delta_{\bar\infty}$. Finally, using equation (\ref{ec:1}), we conclude that $\lambda =\alpha$ and $\mu_1=\mu$. \end{proof}

\begin{proof} [Proof of Theorem \ref{topo:r} (\ref{rome_cyl})] Let $C=[a_1, \ldots, a_n]$ be cylinder in $\Sigma$, and consider $C_{\bar\Sigma} =[a_1, \dots , a_n]_{\bar\Sigma}$. Since $C_{\bar\Sigma}$ has empty boundary (see Remark \ref{rem:clopen}), by Portmanteau's theorem, we have 
\begin{equation*}
\lim_{n \to \infty} \mu_n(C_{\bar\Sigma})= (\lambda \mu + (1- \lambda) \delta_{\bar\infty})( C_{\bar\Sigma}).
\end{equation*}
Finally, note that $\delta_{\bar\infty}( C_{\bar\Sigma})=0$, and that $\mu_n(C_{\bar\Sigma})=\mu_n(C)$ for every $n\in\N$
\end{proof}

\begin{remark}\label{rem:barinfty}
It follows from Theorem \ref{topo:r} that a sequence in $\M(\Sigma,\sigma)$ converges on cylinders to the zero measure if and only if it converges in $\M(\bar\Sigma,\bar\sigma)$ to $\delta_{\bar\infty}$. If such a sequence exists, then $\bar\infty\in\bar\Sigma$. More generally, if there exists a sequence $(\mu_n)_n$ in $\M(\Sigma,{\sigma})$ that converges on cylinders to $\lambda \mu$, where $\lambda\in [0,1)$, then  $\delta_{\bar\infty}\in  \M(\bar\Sigma,{\bar\sigma})$ and $\bar\infty\in\bar\Sigma$.
\end{remark}

The next corollary follows from combining Remark \ref{rem:barinfty} and Theorem \ref{cyr}. 

\begin{corollary}\label{prop:rome}
     The following statements are equivalent:
    \begin{enumerate}
        \item $(\Sigma,\sigma)$ has a finite uniform Rome
        \item There is no sequence in $\M(\Sigma,\sigma)$ that converges on cylinders to the zero measure
        \item $\bar\infty\notin\bar\Sigma$
        \item Any sequence in $\mathcal{M}(\Sigma, \sigma)$ that converges on cylinders to a measure must converge to a probability measure.
    \end{enumerate}
\end{corollary}

\subsection{Finitely additive measures}

It was observed in \cite[Proposition 4.19]{iv} that, in the absence of the $\F$-property, there exist sequences of invariant probability measures converging on cylinders to finitely additive measures that do not extend to countably additive measures. The following result, which is a direct consequence of Theorem \ref{topo:r}, characterizes the sequences of measures that converge on cylinders to finitely additive measures that are not countably additive.

\begin{corollary} \label{nme}
Let $(\Sigma, \sigma)$ be a transitive countable Markov shift. A sequence $(\mu_n)_n$ in $\M(\Sigma,\sigma)$ converges on cylinders to a finitely additive measure that is not countably additive if and only if it converges in the weak$^*$ topology of $\M(\bar\Sigma,\bar\sigma)$ to a measure that is not in the convex hull of $\M(\Sigma,\sigma) \cup \{ \delta_{\bar\infty} \}$. In other words, the limiting measure is not of the form $\lambda \mu + (1-\lambda) \delta_{\bar\infty}$, where $\lambda \in [0,1]$ and $\mu \in \M(\Sigma,\sigma)$.
\end{corollary}

\begin{example} Let $\Sigma_{\text{full}} = \N^\N$ be the full shift. Consider the periodic point $p_n = \overline{1n}$, and let $\mu_n$ denote the periodic measure associated with $p_n$. Observe that $(\mu_n)_n$ converges on cylinders to the finitely additive measure $m$, where $m([1]) = 1/2$ and $m(C) = 0$ for any other cylinder $C$. In this case, the formula $m([1]) = \sum_{s \in\N} m([1,s])$ does not hold. On the other hand, the sequence converges in $\M(\bar\Sigma, \bar\sigma)$ to the periodic measure associated with the periodic point $p_\infty = \overline{1\infty}$. 
\end{example}

\section{Affine geometry of $\M(\bar\Sigma,{\bar\sigma})$ and $\M(\Sigma,\sigma)$}

In this section, we prove Theorem \ref{thm:po} and Theorem \ref{thm:erg}, which describe the affine geometry of \(\mathcal{M}(\bar{\Sigma}, \bar{\sigma})\). As a consequence, we also describe the geometry of $\M(\Sigma,\sigma)$.

\subsection{Compactification of $\M(\Sigma,\sigma)$}
We first prove Theorem \ref{thm:po}, which states that if \((\Sigma, \sigma)\) is a transitive countable Markov shift, then the space of invariant probability measures for the compactification is affine homeomorphic to the Poulsen simplex. This result extends and completes \cite[Theorem 1.2]{iv}, where, under the assumption of the $\F-$property, we showed that the space of invariant sub-probability measures is affine homeomorphic to the Poulsen simplex (see Theorem \ref{sub:comp}). We now extend this result beyond systems satisfying the $\F-$property.

The following fact is a consequence of the inclusion-exclusion principle.

\begin{lemma}\label{lem:inc} Let $\mu\in\M(\bar\Sigma,{\bar\sigma})$ and $A\subseteq \bar\Sigma$ a measurable set satisfying $\mu(A)>\frac{1}{m}$, for some $m\in\N$. Then, for every $h\in \N$, there exists $k\in[h,h+2m)$ such that $$\mu(A\cap \bar\sigma^{-k}A)>\frac{1}{2^{2m}}.$$
\end{lemma} 

The proof of Theorem \ref{thm:po} is based on the following two results. 

\begin{proposition}\label{prop:miss}
Let $(\Sigma, \sigma)$ be a transitive countable Markov shift. The space of periodic measures of $(\bar\Sigma,\bar\sigma)$ is dense in $\M(\bar\Sigma,\bar\sigma)$.
\end{proposition}
\begin{proof} It is well known that $\M(\bar\Sigma, {\bar\sigma})$ is a compact convex set whose extreme points are precisely the ergodic measures. Moreover, by the ergodic decomposition theorem, $\M(\bar\Sigma, {\bar\sigma})$ is a Choquet simplex, and any invariant probability measure can be uniquely represented as an average over the set of ergodic measures. Observe that the set of measures of the form $\mu = \frac{1}{n} (\mu_1 + \ldots + \mu_n)$, where each $\mu_i$ is ergodic, is dense in $\M(\bar\Sigma, {\bar\sigma})$. To complete the proof, it suffices to show that if $\mu = \frac{1}{N} \sum_{i=1}^N \mu_i$, where each $\mu_i$ is ergodic, then $\mu$ can be approximated by periodic measures.

If $(\Sigma, \sigma)$ has a finite uniform Rome, then $\mu_i \neq \delta_{\bar\infty}$ for any $i \in \mathbb{N}$ (see Corollary \ref{prop:rome}). On the other hand, if $(\Sigma, \sigma)$ does not have a finite uniform Rome, then $\delta_{\bar\infty}$ is the weak$^*$ limit of a sequence in $\M(\Sigma, {\sigma})$ (see Remark \ref{rem:barinfty} and Corollary \ref{prop:rome}). It is enough to prove that $\mu$ can be approximated by periodic measures under the assumption that $\mu_i \neq \delta_{\bar\infty}$ for all $1 \leq i \leq N$. By Remark \ref{rem:1}, $\mu_i \neq \delta_{\bar\infty}$ if and only if $\mu_i([\infty]) < 1$.

Choose \( m, M \in \mathbb{N} \) such that \( \mu_i(K_M) > \frac{1}{m} \) for every \( i \in \{1, \ldots, N\} \), where \( K_M = \bigcup_{i=1}^M [i]_{\bar\Sigma} \). Since \( (\Sigma, \sigma) \) is transitive, there exists \( L_0 \in \mathbb{N} \) such that any pair of symbols in \( \{1, \ldots, M\} \) can be connected by an admissible word in \( \Sigma \) of length at most \( L_0 \).

Let $\mathcal{H} = \{f_1, \ldots, f_s\} \subseteq C(\bar\Sigma)$ be a collection of continuous functions on $\bar\Sigma$. Fix $\epsilon > 0$. We will prove that the set 
\begin{equation*}
\Omega= \left\{ \nu \in \M(\bar\Sigma,{\bar\sigma}): \left|\int fd\nu-\int fd\mu \right|<\epsilon, \text{ for every }f\in \mathcal{H} \right\},
\end{equation*}
contains periodic measures. Since sets of the form $\Omega$ are a basis for the weak$^*$ topology on $\M(\bar\Sigma, {\bar\sigma})$, this implies that $\mu$ can be approximated by periodic measures. 

Set $C_0 = \max_{f \in \mathcal{H}} \max_{x \in \bar\Sigma} |f(x)|$. By uniform continuity, there exists $\delta > 0$ such that if $d_\rho(x, y) < \delta$, then $|f(x) - f(y)| <\frac{1}{4} \epsilon$ for every $f\in\mathcal{H}$. Choose $N_0 \in \mathbb{N}$ such that $\frac{1}{2^{N_0}} < \delta$. Note that if $x, y \in \bar\Sigma$ satisfy $x_j = y_j$ for all $j \in \{1, \ldots, N_0\}$, then $d_\rho(x, y) < \delta$, and consequently, $|f(x) - f(y)| <\frac{1}{4} \epsilon$, for every $f\in\mathcal{H}$.
 


 Define 
\begin{equation*}
A^s_{i,\epsilon}= \left\{ x\in \bar\Sigma: \left|\frac{1}{m}\sum_{k=0}^{m-1}f(\sigma^k x)-\int f d\mu_i \right|<\frac{1}{4}{\epsilon},\text{ for every }f\in\mathcal{H} \text{ and }m\ge s \right\}.
\end{equation*}
It follows from Birkhoff ergodic theorem that $\mu_i(A_{i,\epsilon}^s)\to 1$ as $s\to\infty$. Choose $s_0 \in \N$ such that $\mu_i(A_{i,\epsilon}^{s_0})\ge 1-\frac{1}{2^{2m+1}}$ for every $i\in  \{1, \dots,  N  \}$.


Let $n\ge s_0+N_0$. For each $i\in\{1,\ldots,N\}$ there exists $n_i\in [n,n+2m)$ such that $\mu_i(K_M\cap \bar\sigma^{-n_i} (K_M))>\frac{1}{2^{2m}}$ (see Lemma \ref{lem:inc}). Therefore,
$$\mu_i( A^{s_0}_{i,\epsilon}\cap K_M\cap \bar\sigma^{-n_i}(K_M))>\frac{1}{2^{2m+1}},$$
for every $i\in\{1,\ldots,N\}$. Choose a point $x_i\in  A^{s_0}_{i,\epsilon}\cap K_M\cap \bar\sigma^{-n_i}(K_M)$. 

We will construct a periodic point $x(n)$ by gluing orbits of the points $(x_i)^N_{i=1}$. Let ${\bf y_i}$ be the admissible word formed by the first $(n_i+1)$ coordinates of $x_i$. Observe that the first and last letters of ${\bf y_i}$ are in $\{1,...,M\}$. Consider an admissible word of the form ${\bf y}={\bf y_1 w_1y_2w_2...y_Nw_N}$, where ${\bf w_i}$ are admissible words in $\Sigma$ of length at most $L_0$, connecting ${\bf y_i}$ and ${\bf y_{i+1}}$, where ${\bf y_{N+1}}={\bf y_1}$. Define $m_0=0$ and  $m_k=\sum_{i=1}^k \ell({\bf y_iw_i})$ for $k\ge 1$, where $\ell({\bf w})$ is the length of the word ${\bf w}$. Define $x(n)=({\bf yy...})\in \bar\Sigma$ (see Lemma \ref{lem_unir_adm}).  We claim that for sufficiently large $n$, the periodic measure associated to $x(n)$, denoted by $\mu_{x(n)}$, belongs to $\Omega$.

It follows from the construction that 
\begin{enumerate}
\item $|S_{n_k-N_0}f(\bar\sigma^{m_{k-1}}x(n))-S_{n_k-N_0}f(x_k)|\le (n_k-N_0)\frac{1}{4}\epsilon$,
\item $|S_{m_k-n_k+N_0}f(\bar\sigma^{m_{k-1}+n_k-N_0} x(n))|\le (m_k-n_k+N_0)C_0$,
\end{enumerate}
for every $k\in \{1,\ldots,N\}$ and $f\in\mathcal{H}$. Since $x_k\in A_{k,\epsilon}^{s_0}$ and $n_k-N_0\ge s_0$, we know that $$|S_{n_k-N_0}f(x_k)-(n_k-N_0)\int fd\mu_i|<\frac{1}{4}\epsilon (n_k-N_0),$$ for any $f\in\mathcal{H}$.  We conclude that 
$$|S_{m_N}f(x(n))-\sum_{k=1}^N (n_k-N_0)\int f d\mu_k|\le \frac{1}{2}\epsilon \bigg(\sum_{k=1}^N(n_k-N_0) \bigg)+ C_0\sum_{k=1}^N(m_k-n_k+N_0).$$
Note that $m_k - n_k \leq L_0$ and that $\frac{n_k - N_0}{m_N} \to \frac{1}{N}$ as $n \to \infty$. Furthermore, for sufficiently large $n$, we have $|\frac{1}{m_N} S_{m_N} f(x(n)) - \int f  d\mu| < \epsilon$. Consequently, $\mu_{x(n)} \in \Omega$. 
 \end{proof}

\begin{proposition}\label{lem:21} Let $(\Sigma, \sigma)$ be a transitive countable Markov shift.  The set of periodic measures of $(\Sigma,\sigma)$ is dense in $\M(\bar\Sigma,\bar\sigma)$.\end{proposition}
\begin{proof} By Proposition \ref{prop:miss}, it suffices to prove that periodic measures of $(\bar\Sigma,\bar\sigma)$ can be approximated by periodic measures of $(\Sigma,\sigma)$.

 Let $x\in \bar\Sigma$ be a periodic point of order $n$. Suppose that $x\ne \bar\infty$, and further assume that $x_1\in\N$. Then, there are periodic points $(y^k)_k$ in $\Sigma$ of order $n$ such that $\lim_{k\to\infty} y^k_i=x_i$, for every $i\in\N$. The periodic measures associated to the sequence $(y^k)_k$ converges to the periodic measure associated to $x$. If $x=\bar\infty$, then $\Sigma$ does not have a finite uniform Rome and there are sequences of measures in $\M(\Sigma,\sigma)$ that converge to $\delta_{\bar\infty}$ (see Theorem \ref{topo:r} and Corollary \ref{prop:rome}). Since periodic measures in $(\Sigma,\sigma)$ is dense in $\M(\Sigma,\sigma)$ (see Remark \ref{rem:perio}), we can approximate  $\delta_{\bar\infty}$ with periodic measures of $(\Sigma,\sigma)$. \end{proof}

\begin{proof}[Proof of Theorem \ref{thm:po}]
First, note that Proposition \ref{lem:21} implies that $\M(\Sigma,\sigma)$ is dense in $\M(\bar\Sigma, {\bar\sigma})$. Since periodic measures are ergodic and thus extreme points of $\M(\bar\Sigma, {\bar\sigma})$, it follows from Proposition \ref{prop:miss} and \cite[Theorem 2.3]{los} that $\M(\bar\Sigma, {\bar\sigma})$ is affine homeomorphic to the Poulsen simplex. 
\end{proof}

\begin{remark}\label{rem:perio} It follows from Proposition \ref{lem:21} that the set of periodic measures of $(\Sigma,\sigma)$ is dense in $\M(\Sigma, \sigma)$ (see Remark \ref{rem:topextend}).
\end{remark}


\subsection{The set of new ergodic measures}
We now prove Theorem \ref{thm:erg}, which states that if the system satisfies the \(\mathcal{F}-\)property, then the compactification adds only one new ergodic measure: the Dirac measure at the fixed point at infinity. On the other hand, if the system does not satisfy the \(\mathcal{F}-\)property, then the set of ergodic measures in \(\mathcal{M}(\bar\Sigma, \bar\sigma) \setminus \mathcal{M}(\Sigma, \sigma)\) is dense in \(\mathcal{M}(\bar\Sigma, \bar\sigma)\). This reveals a striking dichotomy that highlights the intrinsic nature of the $\F-$property.

\begin{lemma}\label{lem:22} Suppose that $(\Sigma,\sigma)$ does not satisfy the $\F-$property. Then, periodic measures in $\M(\Sigma,{\sigma})$ can be approximated by periodic measures in $\M(\bar\Sigma,\bar\sigma)\setminus \M(\Sigma,\sigma)$.\end{lemma}

\begin{proof}  Let $x\in\Sigma$ be a periodic point of period $n$. Let ${\bf w}=w_1\ldots w_m$ be a $\bar\Sigma$-admissible word such that $w_1=x_1$, $w_m=x_{n}$, and $w_i=\infty$ for some $1<i<m$. The existence of ${\bf w}$ follows from the transitivity of  $(\Sigma,\sigma)$ and the assumption that it  does not satisfy the $\F-$property. Let ${\bf x}=x_1\ldots x_n$ and construct the periodic point $z_k$ associated to the admissible word ${\bf x \ldots x w}$, where the block ${\bf x}$ appears $k$ times. The sequence of periodic measures associated to $(z_k)_k$ converges to the periodic measure associated to $x\in\Sigma$. \end{proof}

\begin{lemma} \label{lema_infi} Let $(\Sigma ,\sigma)$ be a countable Markov shift satisfying the $\F-$property. Then the following statements hold:
\begin{enumerate}
\item  If $(x_1, x_2, \ldots)\in \bar\Sigma$, then there is no $n,m\in\N$ such that $x_n,x_{n+m+1}\in \N$, and $x_{n+k}=\infty$, for $1\le k\le m$.
\item $\delta_{\bar\infty}$ is the unique ergodic measure in $\M(\bar\Sigma,\bar\sigma)\setminus \M(\Sigma,\sigma)$.
 \end{enumerate}
\end{lemma}

\begin{proof}[Proof Lemma \ref{lema_infi}]
We start with the proof of (a). Assume by way of contradiction that there exists a point  $x=(x_1, x_2, \ldots)\in \bar\Sigma$ such that $x_n,x_{n+m+1}\in \N$, and $x_{n+k}=\infty$, for $1\le k\le m$. In particular, the word ${\bf w}=x_n\infty \ldots\infty x_{n+m+1}$ is $\bar\Sigma$-admissible, and therefore there are admissible words ${\bf w_i}=x_n y_1^i\ldots y_m^ix_{n+m+1}$ in $\Sigma$ such that $\lim_{i\to \infty}y_k^i=\infty$, for every $1\le k\le m.$ In particular, there are infinitely many admissible words of length $m+2$ that connect $x_n$ with $x_{n+m+1}$, which contradicts that $\Sigma$ has the $\F-$property. In order to prove (b) observe that if $\mu([\infty]_{\bar\Sigma}\cap \bar\sigma^{-1}[a]_{\bar\Sigma})>0,$ for some $a\in\N$, then by the Birkhoff ergodic theorem almost every point returns infinitely many times to $[\infty]_{\bar\Sigma}$ and $[a]_{\bar\Sigma}$. By part (a) this is not possible. We conclude that if $\mu([\infty]_{\bar\Sigma})>0$, then $\mu([a]_{\bar\Sigma})=0$ for every $a\in\N$, and therefore $\mu=\delta_{\bar\infty}$ (see Remark \ref{rem:1}).\end{proof}

The fact that, for a system satisfying the $\F-$property, the measure \(\delta_{\bar\infty}\) is the only new ergodic measure added by the compactification \((\bar\Sigma, \bar\sigma)\) was previously established for finite entropy systems and for a broader class of metrics in \cite[Proposition 5.1]{gs} and \cite[Lemma 4.7]{it2}.

\begin{proof}[Proof of Theorem \ref{thm:erg}]
It follows by Lemma \ref{lema_infi} that if $(\Sigma,\sigma)$ satisfies the $\F-$property, then $\delta_{\bar\infty}$ is the unique ergodic measure in  $\M(\bar\Sigma,\bar\sigma) \setminus \M(\Sigma,\bar\sigma)$. It remains to prove that if $(\Sigma,\sigma)$ does not satisfy the $\F-$property, then the set of ergodic measures in $\M(\bar\Sigma,\bar\sigma) \setminus \M(\Sigma,\bar\sigma)$ is dense in $\M(\bar\Sigma,\bar\sigma)$. 

Let \( Y \) be the closure of the set of ergodic measures in \( \mathcal{M}(\bar\Sigma, \bar\sigma) \setminus \mathcal{M}(\Sigma, \bar\sigma) \). We claim that \( Y = \mathcal{M}(\bar\Sigma, \bar\sigma) \). By Lemma \ref{lem:22} and Remark \ref{rem:perio}, we have \( \mathcal{M}(\Sigma, \sigma) \subseteq Y \). Furthermore, by Proposition \ref{lem:21}, \( \mathcal{M}(\Sigma, \sigma) \) is dense in \( \mathcal{M}(\bar\Sigma, \bar\sigma) \), which proves the claim.
 \end{proof}

\begin{remark}\label{rem:Fequiv}
      If $\Sigma$ has the $\F-$property, then the  affine map $\Phi:\M(\bar\Sigma,{\bar\sigma})\to\M_{\le1}(\Sigma,\sigma)$ defined by $\Phi(\mu)=\mu$ if $\mu\in\M(\Sigma,\sigma)$, and $\Phi(\delta_{\bar\infty})$ equal to the zero measure is an affine homeomorphism (see Theorem \ref{topo:r} and Lemma \ref{lema_infi}).  
\end{remark}

\subsection{The space invariant probability measures}
A direct consequence of Theorem \ref{thm:po} and Theorem \ref{thm:erg} is that we can describe the spaces of shift invariant probability measures endowed with the weak$^*$ topology.

\begin{theorem}
Let $(\Sigma ,\sigma)$ be a transitive countable Markov shift. Then,
\begin{enumerate}
\item If $(\Sigma, \sigma)$ satisfies the $\F-$property then $\M(\Sigma,\sigma)$ is  affine homeomorphic to a  Poulsen simplex minus a vertex and all of its convex combinations. 
\item  If $(\Sigma, \sigma)$ does not satisfy the $\F-$property then $\M(\Sigma,\sigma)$  is affine homeomorphic to a Poulsen simplex in which we remove a dense set of extreme vertices and all of its convex combinations, but keep another dense set of extreme vertices. 
\end{enumerate}
\end{theorem}

We note that the case in which $(\Sigma, \sigma)$ satisfies the $\F-$property was first obtained in  \cite[Corollary 1.3]{iv}.

\section{The space of ergodic measures of $(\Sigma,\sigma)$}\label{thespaceofmeasures}

In \cite{los}, the authors established several remarkable properties of the Poulsen simplex: they proved it is homogeneous, unique, and universal (see \cite[Theorems 2.3 and 2.5]{los}). Additionally, they showed that its set of extreme points is homeomorphic to \(\ell_2\) (see \cite[Theorem 3.1]{los}). As a consequence, for any transitive subshift of finite type, the space of ergodic measures is homeomorphic to \(\ell_2\). In this section, we prove Theorem \ref{theo_l2}, which states that this result extends to the noncompact setting.

We follow \cite[Section 3]{los} closely  and make use of some results from infinite dimensional topology. Let $Q=[-1,1]^\N$ be the \emph{Hilbert cube} and $P=(-1,1)^\N$ its \emph{pseudo-interior}. It is a theorem of Keller \cite{ke} that any infinite dimensional compact convex set in the Hilbert space $\ell_2$ is homeomorphic to the Hilbert cube. In his construction, Poulsen \cite{pou} introduced a simplex with a dense set of extreme points as a subset of $\ell_2$. In particular, the Poulsen simplex is homeomorphic to $Q$. Furthermore,  Anderson \cite{an} proved that $P$ is homeomorphic to $\ell_2$.  

Let $K\subseteq\ell_2$ be an infinite dimensional compact convex set and set $I^n:=[-1,1]^n$. Denote by $C(I^n,K)$ the space of continuous functions from $I^n$ to $K$ with the compact-open topology. A subset $B\subseteq K$ is said to be a $T-$set if for every $n\in \N$, the set 
$$C_B^n=\{f\in C(I^n,K): f(I^n)\subseteq B\},$$
is dense in $C(I^n,K)$. It is proved in \cite[Corollary 4.3, Chapter IV]{bp} that if $B$ is a subset of the set of extreme points of $K$, a  $T-$set and a $G_\delta-$set, then there exists a homeomorphism $h:Q\to K$ such that $h(P)=B$. The proof of Theorem \ref{theo_l2} reduces to the next two lemmas, where we consider $K=\M(\bar\Sigma,\bar\sigma)$ and $B=\M^e(\Sigma,\sigma)$. Let $\eta$ be a metric on $\M(\bar\Sigma,\bar\sigma)$ compatible with the weak$^*$ topology.

\begin{lemma}\label{lem:Gdelta} The set $\M^e(\Sigma,\sigma)$ is a $G_\delta-$set of $\M(\bar\Sigma,\bar\sigma).$    
\end{lemma}

\begin{proof} For each $k\in\N$ consider the sets $$G_k=\{\mu\in \M(\bar\Sigma,\bar\sigma):\text{ if }\mu=\frac{1}{2}(\mu_1+\mu_2), \text{ where }\mu_1,\mu_2\in\M(\bar\Sigma,\bar\sigma),  \text{ then }\eta(\mu_1,\mu_2)<\frac{1}{k}\},$$
and 
$$J_k=\left\{\mu\in \M(\bar\Sigma,\bar\sigma):\mu([\infty])<\frac{1}{k}\right\}.$$
The sets $G_k$ and $J_k$ are open. In both cases it is easier to verify that their complements are closed. For $G_k$, this follows from the compactness of $\M(\bar\Sigma,\bar\sigma)$.  In the case of $J_k$, we note that $[\infty]$ is a closed subset of $\bar\Sigma$. Consequently, if $(\mu_n)_n$ converges to $\mu$, then 
$$\limsup_{n\to\infty}\mu_n([\infty])\le \mu([\infty]).$$
Set $O_k=G_k\cap J_k$. Since $O_k$ contains the set of periodic measures in $\M(\Sigma,\sigma)$, and this set is dense (see Proposition \ref{lem:21}), we conclude that $O_k$ is a dense open set. Finally, note that $\bigcap_{k\in\N}G_k=\M^e(\bar\Sigma,{\bar\sigma})$, therefore $\bigcap_{k\in\N}O_k=\M^e(\Sigma,\sigma)$ (see Remark \ref{rem:1}). 
\end{proof}

\begin{lemma}\label{lem:Tset} The set $\M^e(\Sigma,\sigma)$ is a $T-$set of $\M(\bar\Sigma,\bar\sigma).$    
\end{lemma}
\begin{proof} We follow the notation introduced in Lemma \ref{lem:Gdelta}. Observe that 
   $$C_{O_k}^n=\{f\in C(I^n,\M(\bar\Sigma,\bar\sigma)): f(I^n)\subseteq O_k \},$$ is an open subset of $C(I^n,\M(\bar\Sigma,\bar\sigma))$, for each $k,n\in\N$. Since $C^n_{\M^e(\Sigma,\sigma)}=\bigcap_{k\in\N} C^n_{O_k},$ the Baire category theorem implies that it suffices to establish the following result:\\\\
   {\bf Claim:} The set $C_{O_k}^n$ is dense in $C(I^n,\M(\bar\Sigma,\bar\sigma))$.\\
   
   {\bf Proof:} We closely follow the proof of \cite[Lemma 3.6]{los}. It is convenient to represent \( I^n \) as the \( n \)-dimensional simplex in \( \mathbb{R}^n \), which we denote by \( \Delta_n \). Let $f:\Delta_n\to \M(\bar\Sigma,\bar\sigma)$ be a continuous map and $\epsilon\in (0,1/k)$. In order to prove the claim, it is enough to construct a continuous map $g:\Delta_n\to \M(\bar\Sigma,\bar\sigma)$ such that $\|g-f\|:=\sup_{x\in \Delta_n}\eta(g(x),f(x))<\epsilon$, and $g(\Delta_n)\subseteq O_k$.\\
   
   Since $f$ is uniformly continuous, there exists $\delta>0$ such that if $|x-y|<\delta$, then $\eta(f(x),f(y))<\epsilon/4$,  where $|\cdot|$ is the euclidean norm in $\Delta_n$. Let $\{\Delta_n^i\}_i$ be a simplicial subdivision of $\Delta_n$ such that the diameter of each $\Delta_n^i$ is smaller than $\delta$, and $\{t_j\}_j$ be the set of vertices of  $\{\Delta_n^i\}_i$. Observe that $\diam f(\Delta_n^i)< \epsilon/4$.
   
Periodic measures of \( (\Sigma, \sigma) \) are dense in \( \mathcal{M}(\bar{\Sigma}, \bar{\sigma}) \) (see Proposition \ref{lem:21}). For each vertex \( t_j \), we select a periodic measure \( \mu_j \) of \( (\Sigma, \sigma) \) such that \( \eta(f(t_j), \mu_j) < \epsilon/4 \). Using the collection of periodic measures \( \{\mu_j\}_j \), we define the map \( g \). Specifically, we set \( g(t_j) = \mu_j \) for each vertex \( t_j \) and extend \( g \) to \( \Delta_n \) by requiring it to be affine on each simplex \( \Delta_n^i \).

By construction, every measure in \( g(\Delta_n) \) is a finite convex combination of periodic measures in \( \mathcal{M}(\Sigma, \sigma) \). Define \( F^i := g(\Delta_n^i) \). By the definition of \( g \), we have \( \eta(g(t_j), f(t_j)) < \epsilon/4 \) for every vertex \( t_j \). Since the diameter of \( F^i \) is attained at its vertices and \( \operatorname{diam} f(\Delta_n^i) < \epsilon/4 \), it follows that \( \operatorname{diam} g(\Delta_n^i) < \epsilon/2 \).  

For any \( x \in \Delta_n^i \), choose a vertex \( t_j \in \Delta_n^i \). Then,  
\[
\eta(f(x), g(x)) \leq \eta(f(x), f(t_j)) + \eta(f(t_j), g(t_j)) + \eta(g(t_j), g(x)) < \epsilon.
\]
Thus, we conclude that \( \|f - g\| < \epsilon \). 
   
It remains to show that \( g(\Delta_n) \subseteq O_k \). Observe that if \( \mu \in F^i \) and \( \mu = \frac{1}{2}(\mu_1 + \mu_2) \), then \( \mu_1, \mu_2 \in F^i \) since \( F^i \) is a face (as ergodic decomposition is unique; for the definition of \emph{face} see \cite[p.387]{j}). Moreover, we have \( \eta(\mu_1, \mu_2) \leq \operatorname{diam} F^i < \epsilon/2 < 1/k \). This implies \( g(\Delta_n^i) \subseteq G_k \) for every \( i \), and consequently, \( g(\Delta_n) \subseteq G_k \).  As noted earlier, the set \( g(\Delta_n) \) consists of probability measures in \( \Sigma \), meaning \( g(\Delta_n) \subseteq J_k \). Combining these results, we conclude that \( g(\Delta_n) \subseteq O_k \).\end{proof}

\begin{proof}[Proof of Theorem \ref{theo_l2}] It follows by Lemma \ref{lem:Gdelta} and Lemma \ref{lem:Tset} that $\M^e(\Sigma,\sigma)$ is  a $T$-set  and a $G_\delta$ subset of $K=\M(\bar\Sigma,\bar\sigma)$. Moreover, $\M^e(\Sigma,\sigma)$ is a subset of the set of extreme points of $K$. It follows by \cite[Corollary 4.3, Chapter IV]{bp} that $\M^e(\Sigma,\sigma)$ is homeomorphic to $P$, and therefore homeomorphic to $\ell_2$ (see \cite{an}). 
\end{proof}

\section{The dual variational principle}\label{sec:dvp}

It is a classical result that, for expanding maps on compact metric spaces, understanding the pressure functional is equivalent to understanding the space of invariant probability measures and their entropy. Indeed, as shown in \cite[Theorem 9.12]{w}, if \( T: X \to X \) is a continuous map on a compact metric space, has finite topological entropy, and the entropy map \( \mu \mapsto h(\mu) \) is upper semi-continuous, then for any \( T \)-invariant probability measure \( \mu \), we have  
\[
h(\mu) = \inf \left\{ P(f) - \int f \, d\mu : f \in C(X) \right\}.
\]  
Here, \( C(X) \) denotes the space of continuous real-valued functions on \( X \), and \( P(f) \) is the topological pressure of \( f \) with respect to the dynamical system \( (X, T) \).

In this section, as an application of our compactification, we prove Theorem \ref{theo:dual}, which extends this result to countable Markov shifts with possibly infinite topological entropy. We replace the finite entropy assumption by the finiteness of the pressure, and the upper semi-continuity of the entropy  map by the upper semi-continuity of the pressure map.


\begin{proof}[Proof of Theorem \ref{theo:dual}]
Observe that for every $g \in C_b(\Sigma)$ we  have
\begin{equation*}
P(\phi+g) - \int g \, d \mu \geq \left(h(\mu) + \int \phi d \mu + \int g\, d \mu\right) -  \int g\, d \mu\geq h(\mu) + \int \phi \, d \mu.
\end{equation*}
Thus,
\begin{equation*}
h(\mu) + \int \phi \, d \mu \leq \inf \left\{P(g+ \phi) - 	\int g\, d \mu : g \in C_b(\Sigma)	\right\}.
\end{equation*}

We proceed to prove the other inequality. Let \( E: \M_\phi(\Sigma,\sigma) \to \mathbb{R} \) be the \emph{free energy function} of $\phi$, defined by
\begin{equation*}
E(\nu)= h(\nu) + \int \phi \, d \nu.
\end{equation*}
It follows from Theorem \ref{thm:usc} that if \( s_\infty(\phi) < 1 \), then the function \( E \) is upper semi-continuous.
Note that $\M_\phi(\Sigma,\sigma)$ is a dense subset of
$\M(\bar\Sigma,\bar\sigma)$ (see Proposition \ref{lem:21}). Let $\bar{E}:\M(\bar\Sigma,\bar\sigma)\to \R \cup\{-\infty\},$
be the map defined by 
\begin{equation*} \label{fr}
\bar{E}(\nu):= \inf \left\{F(\nu): F:\M(\bar\Sigma,\bar\sigma)\to \R, \text{u.s.c.  with } F(\mu) \geq E(\mu) \text{ for } \mu \in \M_\phi(\Sigma,\sigma) \right\}.
\end{equation*}
The map $\bar{E}$ is  upper semi-continuous being the infimum of upper semi-continuous functions. It is also an extension of $E$. The map $\bar{E}$ is sometimes called upper semi-continuous regularization of $E$ (see \cite[Section 6.2, pp.360-363]{b}). We have the following characterization:
\begin{equation*} \label{fr1}
\bar{E}(\nu)= \sup \left\{ \limsup_{n \to \infty} \left( h(\mu_n) + \int \phi \, d \mu_n  \right) : \mu_n \to \nu , \mu_n \in  \M_\phi(\Sigma,\sigma)   \text{ for every } n \in \N    \right\},
\end{equation*}
where the convergence $\mu_n \to \nu$ is in the weak* topology of $\M(\bar\Sigma,\bar\sigma)$. Under our assumptions the map $\bar E$ satisfies the following properties:
\begin{enumerate}
\item \emph{If $\nu \in \M_\phi(\Sigma,\sigma)$, then $E(\nu)=\bar{E}(\nu)$.}\\
 This follows from the upper semi-continuity of $E$.
\item \emph{If $\nu \in  \M(\Sigma,\sigma)$ and $\int \phi \, d \nu=-\infty$, then $\bar{E}(\nu)=-\infty$. }\\
 Indeed, let $(\mu_n)_n$ be a sequence in $\M_\phi(\Sigma,\sigma)$ converging to $\nu$ such that the following holds, $\lim_{n\to\infty}\big(h(\mu_n)+\int \phi d\mu_n\big)=\bar{E}(\nu)$. 
Since $\phi$ is bounded above and the sequence $(\mu_n)_n$ converges in the weak$^*$ topology to $\nu$,  it is a consequence of \cite[Lemma 3]{jmu} that
\begin{equation*}
\limsup_{n \to \infty} \int \phi \, d \mu_n \leq \int \phi \, d \nu.
\end{equation*}
In particular, if $\int \phi \, d \nu = -\infty$, then $\lim_{n \to \infty} \int \phi \, d \mu_n =-\infty$. It follows by \cite[Lemma 3.3.1]{v} that $\lim_{n\to\infty}\big(h(\mu_n)+\int \phi d\mu_n\big)=-\infty$.

\item \emph{ If $\nu \in  \M(\bar\Sigma,\bar\sigma)$ is not in the convex hull of $\M(\Sigma,\sigma)\cup\{\delta_{\bar\infty}\}$, then $\bar{E}(\nu)=-\infty$.} \\
Indeed, let $(\mu_n)_n$ be a sequence in $\M_\phi(\Sigma,\sigma)$ converging to $\nu$ with the property that,  $\lim_{n\to\infty}\big(h(\mu_n)+\int \phi d\mu_n\big)=\bar{E}(\nu)$. If  $\liminf_{n\to\infty}\int \phi d\mu_n=-\infty,$ we use \cite[Lemma 3.3.1]{v} as in the previous case and deduce that $\bar{E}(\nu)=-\infty$. Assume that  $\liminf_{n\to\infty}\int \phi d\mu_n>-\infty$. In this case, it follows by \cite[Theorem 1.1]{v}  that there exists a subsequence of $(\mu_n)_n$ which converges on cylinders to a sub-probability measure. Theorem \ref{topo:r} implies that $\nu$ is in the convex hull of $\M(\Sigma,\sigma)\cup\{\delta_{\bar\infty}\}$, which is a contradiction.
\item \emph{If $\mu\in \M_\phi(\Sigma,\sigma)$ and $\lambda\in[0,1]$, then }
$$\bar E(\lambda\mu+(1-\lambda)\delta_{\bar\infty})=\lambda (h(\mu)+\int \phi d\mu)+(1-\lambda)P_\infty(\phi)$$
This follows from Theorem \ref{topo:r},  Theorem \ref{thm:usc} and parts (a)-(c).
\end{enumerate}
It is a consequence of the above properties  that $\bar{E}$ is an affine function.




Let 
\begin{equation*}
\mathcal{C}= \left\{(\nu, t) \in \M(\bar\Sigma,\bar\sigma) \times  \left(\R \cup \{-\infty\} \right): t \leq 	\bar{E}(\nu)	\right\}.
\end{equation*}
The set $\mathcal{C}$ is convex. Indeed, if $(\nu_1, t_1), (\nu_2, t_2) \in \mathcal{C}$ then there exist sequences $(\rho_n)_n$ and $(\mu_n)_n$  in $\M_\phi(\Sigma,\sigma)$ with $\rho_n \to \nu_1$ and $\mu_n \to \nu_2$ in the weak* topology of  $\M(\bar\Sigma,\bar\sigma)$, such that
\begin{equation*}
t_1 \leq \lim_{n \to \infty} \left(h(\rho_n) + \int \phi \, d\rho_n	\right) \quad \text{ and } \quad
t_2 \leq \lim_{n \to \infty} \left(h(\mu_n) + \int \phi \, d\mu_n	\right). \end{equation*} 
Since the entropy map $h:\M_\phi(\Sigma,\sigma)\to \R$ and the integral are affine maps we have that for $\lambda \in (0,1)$,
\begin{equation*} 
\lambda t_1 + (1- \lambda) t_2 \leq \lim_{n \to \infty}  \left(h(\lambda \rho_n+ (1-\lambda) \mu_n) + \int \phi \, d(\lambda \rho_n+ (1-\lambda) \mu_n)	\right).
\end{equation*}
That is, $\lambda(\nu_1, t_1) +(1-\lambda) (\nu_2, t_2) \in \mathcal{C}$. Furthermore, it follows by the upper semicontinuity of $\bar{E}$ that $\mathcal{C}$ is a closed subset of $\M(\bar\Sigma,\bar\sigma) \times\left(\R \cup \{-\infty\} \right)$.

Let $\mu\in \M_\phi(\Sigma,\sigma)$ and $b >h(\mu) + \int \phi \, d \mu$. Since $\bar{E}$ is upper semi-continuous, it follows that $(\mu, b)$ does not belong to the closure of $\mathcal{C}$.

Define \(\mathcal{D}\) as the convex hull of \(\M_\phi(\Sigma, \sigma) \cup \{\delta_{\bar{\infty}}\}\) if \( P_\infty(\phi) > -\infty \), and set \(\mathcal{D} = \M_\phi(\Sigma, \sigma)\) if \( P_\infty(\phi) = -\infty \). Consider the following  subset of $\mathcal{C}$,
\begin{equation*}
\mathcal{C}_0=\left\{ (\nu, t) \in \mathcal{D}\times \R : t \leq \bar{E}(\nu) \right\}.
\end{equation*}
Note that $\mathcal{C}_0$ is a closed and convex subset of $\M(\bar\Sigma,\bar\sigma) \times \R$. Indeed, suppose that $((\mu_n,t_n))_n$ is a sequence in $\mathcal{C}_0$ that converges to a point $(\nu,t)\in \M(\bar\Sigma,\bar\sigma) \times \R$. Since $t_n\to t$ and $\limsup_{n\to\infty}\bar{E}(\mu_n)\le \bar{E}(\nu)$, we have $t\le \bar{E}(\nu)$. If $\nu\notin  \mathcal{D}$, then $\bar{E}(\nu)=-\infty$, which contradicts that $t\in\R$. The proof of the convexity of $\mathcal{C}_0$ is the same as for $\mathcal{C}$.

It follows from \cite[p.417]{ds} that there exists a continuous linear functional  
\( F: C(\bar\Sigma)^* \times \mathbb{R} \to \mathbb{R} \) such that for every  
\( (\nu, t) \in \mathcal{C}_0 \), we have  $F(\nu, t) < F(\mu, b)$. Furthermore, by \cite[Lemma 8.13]{ew}, there exist \( g \in C(\bar{\Sigma}, \mathbb{R}) \) and \( d \in \mathbb{R} \) such that  
\[
F(\nu, t) = t d + \int g \, d\nu
\]
(see also \cite[Proposition 6]{ph}).
 In particular, for every $\nu \in  \M_\phi(\Sigma,\sigma)$ we have
 \begin{equation*}
 \left(h(\nu) + \int \, \phi d \nu \right)d + \int g \, d \nu <  d b + \int g \, d \mu.
 \end{equation*} 
 If we make $\nu=\mu$, we obtain
  \begin{equation*}
  \left(h(\mu) + \int  \phi \, d \mu \right)d  <  d b.
 \end{equation*} 
 Thus, $d>0$. Therefore,
  \begin{equation*}
h(\nu) + \int  \phi \, d \nu  + \int \frac{g}{d} \, d \nu <   b + \int \frac{g}{d} \, d \mu.
 \end{equation*} 
 From where we have,
 \begin{equation*}
 \inf \left\{	P(\phi+g) - \int  g \, d \mu  : g \in C(\bar\Sigma, \R)		\right\} \leq P\bigg( \phi+\frac{g}{d}\bigg) - \int \frac{g}{d} \, d \mu < b.
 \end{equation*}
 Therefore,
 \begin{equation*}
 h(\mu) + \int  \phi \, d \mu \geq   \inf \left\{	P(\phi+g) - \int  g  \, d \mu  : g \in C(\bar\Sigma, \R)		\right\}.
 \end{equation*}
 This concludes the proof. 
 \end{proof}

  \section{Equilibrium measures}\label{sec:equi}
  
 We present several applications of our results to the study of equilibrium measures in countable Markov shifts.  
  

\subsection{Ergodic measures are unique equilibrium measures}

In this section, we prove that every ergodic measure satisfying a finite entropy assumption is the unique equilibrium measure for a uniformly continuous function. 

Let $(\Sigma, \sigma)$ be a countable Markov shift.
Denote by $\mathcal{B}= \{ [n] : n \in \N \}$ the partition of $\Sigma$ by cylinders of length one. For $\mu \in \M(\Sigma,\sigma)$, let $H_{\mu}(\mathcal{B})= -\sum_{n=1}^{\infty} \mu([n]) \log \mu([n])$ be the entropy of $\mathcal{B}$ with respect to $\mu$. The following is the main result of this sub-section.

\begin{theorem} \label{eem}
Let $(\Sigma, \sigma)$ be a transitive countable Markov shift and $\mu $ an ergodic measure such that   $H_{\mu}(\mathcal{B})$ is finite. Then, there exists $\phi\in\emph{UC}_d(\Sigma)$ such that $\mu$ is the unique equilibrium measure for $\phi$.
\end{theorem}




The construction in the following lemma will be relevant in the proof of Theorem \ref{eem}.

\begin{lemma} \label{finite}
Let $(\Sigma, \sigma)$ be a transitive countable Markov shift and $\mu $ an ergodic measure such that   $H_{\mu}(\mathcal{B})$ is finite.  Then, there exists  $\psi\in  \emph{UC}_d(\Sigma)$ such that $\sup \psi<\infty$, 
$P(\psi)<\infty$ and $\int \psi \, d \mu > -\infty$.
\end{lemma}

\begin{proof}
Let us first assume that $(\Sigma, \sigma)$ is the full-shift on a countable alphabet. Let $r \in \N$ and $\mathcal{B}^r= \{[n_1 \cdots n_r]: (n_1, \dots, n_r) \in \N^r \}$ be the generating partition formed by cylinders of length $r$. Let $(a_n)_n$ be a sequence of strictly positive real numbers such that $\sum_{n=1}^{\infty} a_n <\infty$. Let $\psi: \Sigma \to \R$ be the locally constant function defined by
\begin{equation*}
\psi\Big|_{[n]}=
\begin{cases}
\log \mu([n]) & \text{ if } \mu([n])>0;\\
\log a_n & \text{ if } \mu([n])=0. 
\end{cases}
\end{equation*}
Then $\int \psi \, d \mu =\sum_{n=1}^{\infty} \mu([n]) \log \mu([n]) > -\infty$. Also 
\begin{equation*}
P(\psi)= \log \left( \sum_{n: \mu([n])>0} \mu([n]) + \sum_{n: \mu([n])=0} a_n \right) = \log \left( 1 + \sum_{n: \mu([n])=0} a_n \right)< \infty.
\end{equation*}


Now, let us consider the general case. Assume that  $(\Sigma, \sigma)$ is an arbitrary countable Markov shift and that $\mu \in \M(\Sigma,\sigma)$ satisfies $H_{\mu}(\mathcal{B})<\infty$.  Note that $(\Sigma, \sigma)$ can be regarded as a sub-system of the full-shift, and that $H_{\mu}(\mathcal{B})<\infty$ still holds when $\mu$ is considered as an invariant measure of the full-shift.
Therefore, as we just proved,   there exists a uniformly continuous function  $\psi$, defined on the full shift, such that $P(\psi)<\infty$ and $\int \psi \, d \mu > -\infty$. However, the pressure of $\psi$ with respect to the system $(\Sigma, \sigma)$ is at most the pressure of $\psi$ with respect to the full-shift. The result then follows.
\end{proof}

For a measure $\mu\in \M(\Sigma,\sigma)$ we define 
$$F(\mu)=\left\{\psi\in \text{UC}_d(\Sigma): \sup \psi<\infty, P(\psi)<\infty, \text{ and }\int \psi d\mu>-\infty\right\}.$$
With this definition, we can restate Lemma \ref{finite} as follows: If \( H_\mu(\mathcal{B}) < \infty \), then \( F(\mu) \) is non-empty. The following remark shows that the entropy assumption on the measure is essential for the full shift.

\begin{remark} \label{rai}

If $(\Sigma, \sigma)$ is the full-shift on a countable alphabet the reverse implication in Lemma \ref{finite} holds. More precisely, $F(\mu)$ is non-empty if and only if $H_\mu(\mathcal{B})<\infty$.
 Indeed, in that setting, Mauldin and Urba\'nski \cite[Theorem 2.1.7]{mu} proved the following: If there exists a uniformly continuous function  $\psi:\Sigma \to \R$ such that
$P(\psi)<\infty$ and $\int \psi \, d \mu > -\infty$, then  there exists $q \in \N$ such that $H_{\mu}(\mathcal{B}^q)<\infty$. Note that
\begin{equation*}
H_{\mu}(\mathcal{B}) \leq H_{\mu}(\mathcal{B}^q) \leq q H_{\mu}(\mathcal{B}).
\end{equation*}
The first inequality follows from \cite[Theorem 4.3 (ii)]{w} while the second from sub-additivity  \cite[Corollary 4.9.1]{w}. We conclude that  $H_{\mu}(\mathcal{B})<\infty$. Interestingly, the assumption $h(\mu)<\infty$ does not guarantee the non-emptiness of $F(\mu)$ as there are invariant probability measures on the full shift  such that $H_\mu(\mathcal{B})=\infty$ and $h(\mu)<\infty$. 

For a general countable Markov shift $(\Sigma,\sigma)$ the non-emptiness of $P(\mu)$ is 
 not characterized by the finiteness of $H_\mu(\mathcal{B})$. Indeed, Gurevich \cite[Section 3]{gu} constructed an ergodic invariant measure $\mu$ for a finite entropy countable Markov shift $(\Sigma, \sigma)$ (the renewal shift, see \cite[Section 4]{s2}) such that
$h(\mu)< \infty$ and $H_{\mu}(\mathcal{B})=\infty$. 
Let $c \in \R$ and $\phi:\Sigma \to \R$ be a constant function $\phi=c$. Thus, since the system in Gurevich's example has finite entropy, we have that $P(\phi)< \infty$,
$\int \phi \, d \mu=c$ and $H_{\mu}(\mathcal{B})=\infty$. 
\end{remark}

Let $\psi:\Sigma \to \R$ be the function provided by Lemma \ref{finite}. We can assume that $s_\infty(\psi)<1$; otherwise replace $\psi$ by $t\psi$ for $t>1$. 
Consider, as in the proof of Theorem \ref{theo:dual}, the free energy function of $\psi$.  That is, the upper semi-continuous function  $$\bar E:\M(\bar\Sigma,\bar\sigma) \to \R \cup \{-\infty\}$$  which is defined by $\bar E(\nu)= h(\nu) + \int \psi \, d \nu,$ for measures $\nu\in \M_\psi(\Sigma,\sigma)$. In order to prove Theorem \ref{eem} we will adapt the strategy proposed by Phelps \cite{ph}. 

Let $f_1:\M(\bar\Sigma,\bar\sigma) \to \R \cup \{\infty\}$ be defined by
\begin{equation*}
f_1(\nu)=
\begin{cases}
-\bar E(\mu) & \text{ if } \nu=\mu; \\
-P(\psi) & \text{ if } \nu \neq \mu.
\end{cases}
\end{equation*}
The function \( f_1 \) is convex and upper semi-continuous. Define \( f_2 = -\bar{E} \). Note that \( f_2 \) is lower semi-continuous and affine; in particular, it is concave. Moreover, we have \( f_1(\mu) = f_2(\mu) \) and, more generally, \( f_1(\nu) \leq f_2(\nu) \). The following result is a special case of \cite[Théorème]{ed} (see also \cite[Theorem II.3.10]{al}).

\begin{lemma}[Edwards]
There exists an affine continuous function $l_1:\M(\bar\Sigma,\bar\sigma) \to \R$ such that
$f_1 \leq l_1 \leq f_2$.
\end{lemma}

Observe that \((l_1 + \bar{E})(\mu) = 0\) and that for any \(\nu \in \M(\bar{\Sigma}, \bar{\sigma})\), we have \((l_1 + \bar{E})(\nu) \leq 0\). By \cite[Lemma 4]{ph}, there exists an affine continuous function \( l_2 : \M(\bar{\Sigma}, \bar{\sigma}) \to \mathbb{R} \) such that \( l_2(\mu) = 0 \) and \( l_2(\nu) < 0 \) for all \(\nu \in \M(\bar{\Sigma}, \bar{\sigma}) \setminus \{\mu\} \). Consider the affine continuous function \( l = l_1 + l_2 \). It follows that \( (l + \bar{E})(\mu) = 0 \) and \( (l + \bar{E})(\nu) < 0 \) for every \( \nu \neq \mu \). By \cite[Proposition 6]{ph}, there exists a continuous function \( g: \bar{\Sigma} \to \mathbb{R}\) such that \( l(\nu) = \int g \, d\nu \).

\begin{proof}[Proof of Theorem \ref{eem}]
We can conclude the proof noting that $\phi= \psi + g$ satisfies the required properties. Indeed, for every $\nu \in \M_\psi(\Sigma,\sigma)=\M_\phi(\Sigma,\sigma)$ with $\nu \neq \mu$ we have
\begin{equation*}
0>(l + \bar E)(\nu)= \int g \, d \nu + \int \psi \, d \nu + h(\nu).
\end{equation*}
On the other hand,
\begin{equation*}
0=(l + \bar E)(\mu)= \int g \, d \mu + \int \psi \, d \mu + h(\mu).
\end{equation*}
That is, the measure $\mu$ is the unique equilibrium measure for $\phi$.
\end{proof}

\begin{remark}
In view of Theorem \ref{eem}, a natural question is whether given a face $J$  of  $\M(\Sigma,\sigma)$ there exists a continuous function $\phi:\Sigma \to \R$ for which the set of equilibrium measures for $\phi$ is exactly $J$. An observation by Phelps (see \cite[Remark 9]{ph}) allows us to answer this question: 
If $J$ is a face of  $\M(\Sigma,\sigma)$ so that the pressure is continuous in the  face in $\M(\bar \Sigma,\bar \sigma)$ that contains $J$,  then the required continuous function exists  \footnote{It seems that the additional assumption of continuity of the pressure in the face,  in Theorems 5 and 6 in \cite{j}, is required  for the outlined argument in that proof to hold.}.
\end{remark}

\subsection{The set of measures which are not equilibrium measures is dense} \label{neq}

The following result addresses a complementary question to the one considered in the previous section and is based on an observation by Walters.

\begin{theorem} \label{no-eq}
Let $(\Sigma, \sigma)$ be a transitive countable Markov shift. Then, the set of measures that are not equilibrium measures for any bounded above continuous function is dense in $\M(\Sigma,\sigma)$. 
\end{theorem}

\begin{proof} Let $(p_n)_n$ be a sequence such that $p_n>0$ for every $n  \in \N$ and $\sum_{n=1}^{\infty} p_n=1$. Consider a sequence $(\mu_n)_n$  in $\M(\Sigma,\sigma)$ of finite entropy measures which converges in the weak* topology to a measure $\nu\in\M(\Sigma,\sigma)$ satisfying that $\lim_{n \to \infty} h(\mu_n) < h(\nu)<\infty$.  Walters \cite[p.475]{w2} observed that the measure $\mu=\sum_{n=1}^{\infty} p_n \mu_n$ is not an equilibrium measure for any continuous function.
Indeed, assume by way of contradiction that $\mu$ is an equilibrium measure for $\phi$. That is,
\begin{eqnarray*}
P(\phi)=h(\mu) + \int \phi \, d \mu= \sum_{n=1}^{\infty} p_n h(\mu_n)  +\sum_{n=1}^{\infty} p_n \int \phi \, 
d\mu_n= \sum_{n=1}^{\infty} p_n \left(h(\mu_n) + \int \phi \,d \mu_n \right).
\end{eqnarray*}
Thus, for every $n \in \N$ we have $P(\phi)=h(\mu_n) + \int \phi \, d \mu_n.$ Since $P(\phi)=\lim_{n \to \infty} (h(\mu_n) +  \int \phi \, d \mu_n)$, and $\limsup_{n\to\infty}h(\mu_n)<\infty$, we have that $\lim_{n \to \infty}\int \phi \, d \mu_n >-\infty$. By \cite[Lemma 3]{jmu} we note that $\lim_{n \to \infty}\int \phi \, d \mu_n \leq \int \phi \, d \nu$, and therefore $\int \phi d \nu>-\infty$. Finally, observe that $P(\phi)=\lim_{n \to \infty} (h(\mu_n) +  \int \phi \, d \mu_n) < h(\nu) + \int \phi \, d \nu$, which is a  contradiction.

Note that the set of periodic measures and the set of positive entropy measures are both dense in $\M(\Sigma,\sigma)$ (see Remark \ref{rem:perio}). It is enough to prove that any measure $\nu\in \M(\Sigma,\sigma)$ with positive entropy can be approximated by measures that are not equilibrium states. Let $\epsilon>0$ and consider a sequence of periodic measures $(\mu_n)_n$ which are $\epsilon$-close to $\mu$. Let  $(p_n)_n$ be a vector with $p_n>0$ for every $n  \in \N$ and $\sum_{n=1}^{\infty} p_n=1$. The measure $\sum_{n=1}^{\infty} p_n \mu_n$  is not an equilibrium measure and it is $\epsilon$-close to $\nu$.  This proves the claim.
\end{proof}

\end{document}